\documentclass[11pt]{amsart}
\usepackage{lmodern}
\usepackage{amsmath, amsthm, amssymb, amsfonts}
\usepackage[normalem]{ulem}
\usepackage{hyperref}

\usepackage{mathrsfs}

\usepackage{verbatim} 
\usepackage{longtable}

\usepackage{mathtools}

\usepackage{tikz}
\usetikzlibrary{decorations.pathmorphing}
\tikzset{snake it/.style={decorate, decoration=snake}}

\usepackage{caption}

\usepackage{tikz-cd}
\usetikzlibrary{arrows}

\theoremstyle{plain}
\newtheorem{thm}{Theorem}[section]
\newtheorem{cor}[thm]{Corollary}
\newtheorem{lem}[thm]{Lemma}
\newtheorem{prop}[thm]{Proposition}
\newtheorem{conj}[thm]{Conjecture}

\theoremstyle{definition}

\theoremstyle{remark}
\newtheorem{rmk}[thm]{Remark}

\newcommand{\BA}{{\mathbb{A}}}

\newcommand{\BC}{{\mathbb{C}}}
\newcommand{\BD}{{\mathbb{D}}}

\newcommand{\BH}{{\mathbb{H}}}

\newcommand{\BN}{{\mathbb{N}}}

\newcommand{\BP}{{\mathbb{P}}}
\newcommand{\BQ}{{\mathbb{Q}}}

\newcommand{\BT}{{\mathbb{T}}}

\newcommand{\BZ}{{\mathbb{Z}}}

\newcommand{\CC}{{\mathcal C}}

\newcommand{\CE}{{\mathcal E}}
\newcommand{\CF}{{\mathcal F}}

\newcommand{\CH}{{\mathcal H}}

\newcommand{\CJ}{{\mathcal J}}
\newcommand{\CK}{{\mathcal K}}
\newcommand{\CL}{{\mathcal L}}
\newcommand{\CM}{{\mathcal M}}

\newcommand{\CO}{{\mathcal O}}
\newcommand{\CP}{{\mathcal P}}

\newcommand{\CS}{{\mathcal S}}

\newcommand{\CX}{{\mathcal X}}
\newcommand{\CY}{{\mathcal Y}}
\newcommand{\CZ}{{\mathcal Z}}

\newcommand{\sslash}{\mathbin{/\mkern-6mu/}}

\DeclareFontFamily{OT1}{rsfs}{}
\DeclareFontShape{OT1}{rsfs}{n}{it}{<-> rsfs10}{}
\DeclareMathAlphabet{\curly}{OT1}{rsfs}{n}{it}

\newcommand{\Coh}{\mathrm{Coh}}

\usepackage{tikz}
\usepackage{lmodern}
\usetikzlibrary{decorations.pathmorphing}

\begin{document}
\title[Cohomological $\chi$-independence]{Cohomological $\chi$-independence for moduli of one-dimensional sheaves and moduli of Higgs bundles}
\date{\today}

\author[D. Maulik]{Davesh Maulik}
\address{Massachusetts Institute of Technology}
\email{maulik@mit.edu}

\author[J. Shen]{Junliang Shen}
\address{Yale University}
\email{junliang.shen@yale.edu}

\begin{abstract}
We prove that the intersection cohomology (together with the perverse and the Hodge filtrations) for the moduli space of one-dimensional semistable sheaves supported in an ample curve class on a toric del Pezzo surface is independent of the Euler characteristic of the sheaves. We also prove an analogous result for the moduli space of semistable Higgs bundles with respect to an effective divisor $D$ of degree $\mathrm{deg}(D)>2g-2$. Our results confirm the cohomological $\chi$-independence conjecture by Bousseau for $\BP^2$, and verify Toda's conjecture for Gopakumar--Vafa invariants for certain local curves and local surfaces.

For the proof, we combine a generalized version of Ng\^o's support theorem, a dimension estimate for the stacky Hilbert--Chow morphism, and a splitting theorem for the morphism from the moduli stack to the good GIT quotient.
\end{abstract}

\maketitle

\setcounter{tocdepth}{1} 

\tableofcontents
\setcounter{section}{-1}

\section{Introduction}

We work over the complex numbers $\BC$.

\subsection{Cohomological $\chi$-independence}\label{Sec0.1} Let $C$ be a nonsingular irreducible projective curve of genus $g \geq 2$. The moduli space $N_{n,\chi}$ of (slope-)semistable vector bundles $\CE$ with
\[
\mathrm{rank}(\CE) = n, \quad \chi(\CE) = \chi
\]
is an irreducible projective variety whose topology has been studied intensively for decades. When we fix the rank $n$, tensor product and duality induce natural isomorphisms between the moduli spaces indexed by different Euler characteristics (or degrees)
\begin{equation}\label{0.1_1}
N_{n,\chi} \simeq N_{n, \chi+n}, \quad N_{n,\chi} \simeq N_{n, (2-2g)n-\chi}.
\end{equation}
Under the assumption $\mathrm{gcd}(n,\chi)=1$ so that the moduli spaces $N_{n,\chi}$ are nonsingular, Harder--Narasimhan proved in  \cite[Theorem 3.3.2]{HN} that the Poincar\'e polynomials of $N_{n,\chi}$ are distinct unless the moduli spaces are related via (\ref{0.1_1}).  

In this paper, we are interested in moduli spaces where the cohomological information does \emph{not} depend on the Euler characteristic $\chi$. More precisely, we consider the following two types of moduli spaces $M^L_{\beta,\chi}$ and $\widetilde{M}_{n,\chi}$:
\begin{enumerate}
    \item[(A)] $M^L_{\beta, \chi}$ is the moduli space of 1-dimensional semistable sheaves $\CF$ with 
    \[
    [\mathrm{supp}(\CF)] = \beta, \quad \chi(\CF) =\chi
    \]
    on a nonsingular \emph{toric} del Pezzo surface $S$. Here the semistability is with respect to a polarization $L$ on $S$, $\mathrm{supp}(-)$ denotes the Fitting support, and $\beta$ is an ample curve class.
    \item[(B)] $\widetilde{M}_{n,\chi}$ is the moduli space of semistable Higgs bundles $(\CE, \theta)$ with respect to an effective divisor $D$ of degree $\mathrm{deg}(D)>2g-2$ on $C$ with 
    \[
    \mathrm{rank}(\CE) = n,\quad \chi(\CE) = \chi.
    \]
\end{enumerate}
We refer to Section \ref{Sec2} for more details on these moduli spaces.   When $\chi$ is chosen so that there are no strictly semistable objects, the moduli spaces $M^L_{\beta,\chi}$ and $\widetilde{M}_{n,\chi}$ are nonsingular, and we consider their singular cohomology.  However, for arbitrary values of $\chi$, these moduli spaces can be singular, due to the presence of strictly semistable objects.  In this case, it is more natural for us to study their \emph{intersection cohomology}.  Our main result states that, unlike the case of curves, the intersection cohomology of these spaces is independent of the choice of $\chi$:

\begin{thm}\label{thm0}
For any $\chi, \chi' \in \BZ$, there are isomorphisms of graded vector spaces
\[
\mathrm{IH}^*(M^L_{\beta, \chi}) \simeq \mathrm{IH}^*(M^L_{\beta, \chi'}), \quad \mathrm{IH}^*(\widetilde{M}_{n, \chi}) \simeq \mathrm{IH}^*(\widetilde{M}_{n, \chi'}),
\]
where $\mathrm{IH}^*(-)$ denotes the intersection cohomology.  Moreover, these isomorphisms respect perverse and Hodge filtrations carried by these vector spaces. 
\end{thm}

This phenomenon is surprising, since there is no direct geometric relationship other than those parallel to (\ref{0.1_1}) between these moduli spaces with different Euler characteristics, and the result applies to both smooth and singular moduli spaces.  For example in the case (B), the moduli space is nonsingular if and only if $\mathrm{gcd}(n,\chi) =1$.   Nevertheless, the result on intersection cohomology holds uniformly.  Regarding the second part of the theorem and compatibility with filtrations, see Theorem \ref{thm0.2} for further refinements.

Theorem \ref{thm0} proves the cohomological $\chi$-independence conjecture \cite[Conjecture 0.4.3]{PB} of the moduli space of 1-dimensional semistable sheaves on $\BP^2$, which further proves \cite[Conjecture 0.4.2]{PB2} on the BPS numbers of the log K3 surface $(\BP^2, E)$; see \cite[Theorem 0.4.5]{PB2}. Its refinement (Theorem \ref{thm0.2}) proves Toda's conjecture \cite[Conjecture 1.2]{Toda} on the Gopakumar--Vafa invariants in the cases of certain local curves and local toric del Pezzo surfaces with ample curve classes; see Theorem \ref{thm0.5}. In case (A), when $S = \BP^2$, it was proven by Bousseau \cite[Theorem 0.5.2]{PB2} that the dependence of the (intersection) Betti numbers on $\chi$ only relies on $\mathrm{gcd}(\mathrm{deg}(\beta), \chi)$, using connections with Gromov--Witten theory for the log K3 surface $(\BP^2,E)$ and scattering diagrams. In case (B), when $\mathrm{gcd}(n,\chi) =1$, the equality of Poincar\'e polynomials was proved by a direct calculation in work of Mozgovoy-Schiffmann \cite{MS} and Mellit \cite{Mellit}, as well as in Groechenig-Wyss-Ziegler \cite{GWZ} by $p$-adic integration.  We discuss connections between our theorems and enumerative geometry in Section \ref{Sec0.3} in more detail.

\begin{rmk}
By \cite{delPezzo}, a nonsingular del Pezzo surface belongs to one of the following types:
\begin{enumerate}
    \item[(a)] $\BP^2$.
    \item[(b)] $\BP^1 \times \BP^1$.
    \item[(c)] The blow-up of $\BP^2$ at $n$ very general points with $1\leq n\leq 8$.   
\end{enumerate}
Hence Theorem \ref{thm0} recovers the case when a del Pezzo surface belongs to (a,b), or (c) with $n \leq 3$. We note that the Fano condition is essential (see Section \ref{CY}), but the toric condition is due to a techinical result (Proposition \ref{prop2.5}) which we expect to hold for all del Pezzo surfaces. In other words, if the inequality (\ref{eqn_thm2.5}) of Proposition \ref{prop2.5} is proven for any del Pezzo surface $S$, then Theorem \ref{thm0} (as well as Theorem \ref{thm0.2} below) also holds for any del Pezzo surface.
\end{rmk}
\begin{rmk}
Although we will not require it further, our proof of Theorem \ref{thm0} actually provides a natural isomorphism between these spaces, well-defined up to a scalar, which is compatible with the perverse and Hodge filtrations.
\end{rmk}

\subsection{A support theorem}
Comparing to $N_{n,\chi}$, a key feature of a moduli space $M$ of (A) or (B) is that it admits a morphism $h: M \to B$ that behaves like a completely integrable system.  Here $h$ is the Hilbert--Chow morphism 
\begin{equation}\label{0.2_1}
h: M^L_{\beta,\chi} \to B:=\BP H^0(S ,\CO_S(\beta)), \quad \CF \mapsto  \mathrm{supp}(\CF)
\end{equation}
in the case (A), and the Hitchin fibration
\begin{equation}\label{0.2_2}
h: \widetilde{M}_{n, \chi} \to B:=\bigoplus_{i=1}^nH^0(C, \CO(iD)), \quad (\CE, \theta) \mapsto \mathrm{char}(\theta)
\end{equation}
in the case (B). In either case, there is a maximal Zariski open subset $U \subset B$ parameterizing nonsingular curves in the linear system $|\beta|$ or nonsingular spectral curves over $C$. We denote by $\pi: \CC \to U$ the smooth map given by the universal curve over $U$. 

\begin{thm}\label{thm0.2}
Let $M$ be a moduli space of (A) or (B), and let $h: M \to B$ be the morphism given by (\ref{0.2_1}) or (\ref{0.2_2}) respectively. Let $\pi: \CC \to U \subset B$ be the universal curve of genus $d$. Then there is an isomorphism
\begin{equation}\label{0.2_3}
    Rh_* \mathrm{IC}_M   \simeq  \bigoplus_{i=0}^{2d} \mathrm{IC}\left( \wedge^i R^1\pi_* \BQ_\CC \right)[-i+d]
\end{equation}
in the bounded derived category $D^b\mathrm{MHM}(B)$ of mixed Hodge modules on $B$.
\end{thm}
   
Since the righthand side of (\ref{0.2_3}) clearly does not depend on $L$ or $\chi$, Theorem \ref{thm0.2} implies Theorem \ref{thm0} immediately by taking global cohomology. The sheaf-theoretic nature of (\ref{0.2_3}) further yields refinements of Theorem \ref{thm0} involving perverse and Hodge filtrations. 

Although Theorem \ref{thm0.2} concerns mixed Hodge modules, it suffices to work with perverse sheaves for the proof. In fact, it is not difficult to check (\ref{0.2_3}) over $U$:
\begin{equation}\label{0.2_4}
Rh_* \BQ_{h^{-1}(U)} \simeq \bigoplus_{i=0}^{2d}\wedge^i R^1\pi_* \BQ[-i]
\end{equation}
which only concerns the variation of Hodge structures of abelian varieties; see Proposition \ref{prop2.2}. In view of the decomposition theorem \cite{Saito} for Hodge modules, to prove (\ref{0.2_3}) from (\ref{0.2_4}), we only need to verify that every semi-simple component of $Rh_*\mathrm{IC}_M$ has \emph{full support} $B$. This can be checked completely via the decomposition theorem \cite{BBD} of $Rh_*\mathrm{IC}_M$ in terms of (shifts of) semi-simple perverse sheaves. In particular, Theorem \ref{thm0.2} can be viewed as a \emph{support theorem} for the moduli spaces (A) and (B).

Ng\^o introduced \emph{a support theorem} in \cite{Ngo} which determines the supports of the direct image complex $Rf_*\BQ$ for certain morphism $f: M \to B$ called weak abelian fibration. It played a crucial role in his proof of the fundamental lemma of the Langlands program. After that, support theorems become powerful tools in various branches of mathematics; see for example \cite{dCHM1, dCRS, MS, MY, MiSh, Yun3, YZ}.

In our proof of Theorem \ref{thm0.2}, we systematically develop techniques for applying Ng\^o's support theorem to a more general setup. More precisely, we do not assume that the total space $M$ is nonsingular, and we work with more general objects $\CK \in D^b_c(M)$ than the trivial local system $\BQ$ on $M$. Theorem \ref{supp_thm} reduces a support inequality of Ng\^o type to a \emph{relative dimension bound} (see the condition (c)) for the complex $Rf_*\CK$. Then we introduce techniques to check this bound when $M$ is a moduli space of (A) or (B), and $\CK$ is the intersection cohomology complex $\mathrm{IC}_M$.

\subsection{Enumerative geometry}\label{Sec0.3}
The cohomological $\chi$-independence phenomenon is expected to be part of a much more general phenomenon in the context of enumerative geometry of curves on Calabi--Yau 3-folds, specifically the proposal for Gopakumar-Vafa invariants developed in \cite{MT} and \cite{Toda}.

Let $X$ be a Calabi--Yau 3-fold with $\beta \in H_2(X, \BZ)$ a curve class, and let $\sigma \in \mathrm{Pic}(X)_\BC$ be an element in the complexified ample cone of $X$. Following Davison--Meinhardt \cite{DM1}, and
conditional on the conjectural existence of a certain orientation, Toda introduced in \cite{Toda} the \emph{BPS sheaf} 
\[
\phi_{\mathrm{BPS}} \in  \mathrm{Perv}(M^\sigma_{\beta,\chi})
\]
which is a perverse sheaf on the moduli space $M^\sigma_{\beta,\chi}$ of $\sigma$-semistable sheaves on $X$. Consider the Hilbert--Chow map
\[
h: M^\sigma_{\beta,\chi} \to \mathrm{Chow}_{\beta}(X), \quad \CF \mapsto \mathrm{supp}(\CF).
\] 
For any $\gamma \in \mathrm{Chow}_{\beta}(X)$, the \emph{Gopakumar--Vafa (GV)} invariant (see \cite[Definition 1.1]{Toda})
is defined by the identity:
\begin{equation}\label{GV}
\Phi_{\sigma}(\gamma, \chi): = \sum_{i\in \BZ}\chi\left( {^\mathfrak{p}}\CH^i(Rh_* \phi_{\mathrm{BPS}})|_{\gamma} \right)y^i \in \BZ[y, y^{-1}].
\end{equation}

If $\chi$ is chosen so there are no strictly semistables, this definition specializes to the definition of Gopakumar-Vafa invariants in Maulik-Toda \cite{MT}.   For any choice of $\chi$ and $\sigma$, these invariants are conjectured to encode the same information as the Gromov-Witten invariants of $X$ in the curve class $\beta$ and \emph{arbitrary genus}.  Since the latter invariants are independent of $\chi$ and $\sigma$, in order for this conjecture to be well-posed, the Gopakumar-Vafa invariants should be independent of this extra data as well.  More precisely, Toda made in \cite[Conjecture 1.2]{Toda} the following  conjecture concerning the structure of GV invariants, extending \cite[Conjecture 3.3]{MT}.


\begin{conj}[Toda]\label{Conj}
The invariant (\ref{GV}) is independent of $\sigma$ and $\chi$.
\end{conj}

The invariants (\ref{GV}) specialize to a certain case of the Joyce--Song generalized Donaldson-Thomas (DT) invariants \cite{JS}, and Conjecture \ref{Conj} is expected to refine the Joyce--Song conjecture \cite[Conjecture 6.20]{JS} on the generalized DT invariants, which in turn implies the strong rationality conjecture for Pandharipande-Thomas invariants \cite{PT, Toda2}.

Although currently the existence of the BPS sheaf is conjectural for most cases, it is known to exist for local curve and surface geometries;  it was proven by Meinhardt \cite[Theorem 1.1]{Mein} that when $X$ is a local curve $\mathrm{Tot}_C(\CO_C(D) \oplus K_C(-D))$ with $\mathrm{deg}(D)>2g-2$ or a local del Pezzo surface $\mathrm{Tot}(K_S)$, the BPS sheaf coincides with the intersection cohomology complex of the moduli space. 

\begin{thm}\label{thm0.5}
Conjecture \ref{Conj} holds when $X$ is a local curve $\mathrm{Tot}_C(\CO_C(D) \oplus K_C(-D))$ with $D$ effective of $deg(D)>2g-2$ and $\beta = n[C]$, or a local toric del Pezzo surface $\mathrm{Tot}(K_S)$ and $\beta$ is an ample curve class on $S$.
\end{thm}

In fact, Toda showed that (\ref{GV}) is independent of the stability parameter $\sigma$ under certain conditions which hold for local curves and local surfaces \cite[Theorem 7.3]{Toda}. Hence we may assume that $\sigma$ is given by a rational polarization, and the $\chi$-independence of (\ref{GV}) follows from the $\chi$-independence of the complex $Rf_* \mathrm{IC}_M$ for a moduli space of (A) or (B) with the Hilbert--Chow map (\ref{0.2_1}) and (\ref{0.2_2}) respectively. The latter is given by Theorem \ref{thm0.2}.

Cohomological $\chi$-independence for Higgs bundles has also been studied systematically with connections to Kac-polynomials and quivers. We refer to \cite{Kac} and the references therein for more details.  
For contractible curves on Calabi-Yau threefolds, cohomological $\chi$-independence has been studied by Davison \cite{D4} who has proposed a representation-theoretic approach in that case via the cohomological Hall algebra.  For other places where GV invariants arise geometrically, see  \cite{SY} and \cite{CDP} for connections with hyper-K\"ahler geometries and the $P=W$ conjecture \cite{dCHM1, dCMS} respectively.

\subsection{K3 surfaces and O'Grady 10}\label{CY} As illustrated in the following example of K3 surfaces, the ``Fano" condition for the surface $S$ in (A) and the condition ``$\mathrm{deg}(D)>2g-2$" for Higgs bundles in (B) are essential for the $\chi$-independence to hold for intersection cohomology groups.

Let 
\[
(S, L), \quad L =\CO_S(\beta), \quad \beta^2 =2
\]
be a general polarized K3 surface of degree 2.  The linear system $|\beta|$ is 2-dimensional whose general member is a genus 2 nonsingular curve. The linear system $|2\beta|$ is 5-dimensional. We consider the moduli space of semistable sheaves on $S$ supported in the curve class $2\beta$.

 If $\chi =1$, the moduli space $M^L_{2\beta, 1}$ is nonsingular and deformation equivalent to the Hilbert scheme of 5 points on a K3 surface. When $\chi = 0$, the moduli space $M^L_{2\beta, 0}$ is singular which admits a symplectic resolution. The resolved variety provides O'Grady's 10-dimensional ``sporadic" example of compact hyper-K\"ahler manifolds. As a key step in their analysis of the topology of the O'Grady 10 variety, de Cataldo--Rapagnetta--Sacc\`a \cite{dCRS} studies the fibrations (\ref{0.2_1}):

 \begin{equation*}
    \begin{tikzcd}[column sep=small]
    M^L_{2\beta, 0} \arrow[dr, "h_0"]  & & M^L_{2\beta, 1} \arrow[dl, "h_1"] \\
       & {|2\beta|}, &  
\end{tikzcd}
\end{equation*}
where $h_i(\CF) = \mathrm{supp}(\CF)$. Combining Corollary 3.6.5 and Proposition 4.7.2 of \cite{dCRS}, we observe that
\begin{equation}\label{eqnnn6}
Rh_{1*} \mathrm{IC} = Rh_{0*} \mathrm{IC} \oplus \mathscr{S}[3]
\end{equation}
where $\mathscr{S}$ is a semisimple object supported on the divisor $\mathrm{Sym}^2(|\beta|)  \subset |2\beta|$. Furthermore, by \cite[Proposition 4.6.1]{dCRS}, the object $\mathscr{S}$ (which is denoted by $\mathscr{S}^-_\Sigma$ in \cite{dCRS}) has non-trivial global cohomology. Hence we see from (\ref{eqnnn6}) that the $\chi$-independence fails for the K3 surface $S$ both sheaf theoretically (Theorem \ref{thm0.2}) and cohomologically (Theorem \ref{thm0}).

A similar phenomenon as above is expected to hold for the Higgs bundles with $D = K_C$.

Failure of the $\chi$-independence for the ``Calabi--Yau" case is due to the fact that the BPS sheaf is different from the intersection cohomology complex on the moduli space.

\subsection{Plan of the paper}
In Section 1, we formulate and prove a generalized version of Ng\^o's support theorem, which applies to singular varieties and more general complexes. In order to apply this support theorem to intersection cohomology complexes, we need to prove a bound for IC-complexes (which holds automatically in the smooth case). This is accomplished in Sections 2 and 3 where we combine techniques from algebraic stacks, nilpotent Higgs bundles, moduli of framed objects, and unbounded complexes. Then in Section 4, we follow a strategy of Chaudouard--Laumon to show that the support inequalities are sufficient to deduce our theorems for moduli of 1-dimensional sheaves and Higgs bundles.

\subsection{Acknowledgement}
We are grateful to Bhargav Bhatt, Pierrick Bousseau, and Johan de Jong for helpful discussions and Pinka and Peter Pinkerton for further assistance. We would like to thank Pierrick Bousseau and Tudor Padurariu for their careful reading of an early draft of the paper and pointing out several typos. We also thank the anonymous referee for careful reading and numerous useful suggestions. J.S. was supported by the NSF grant DMS-2134315.

\section{A support theorem for self-dual complexes} \label{Section_Supp}

\subsection{Overview}\label{sec1.1}
The purpose of this section is to formulate and prove a generalized version of Ng\^o's support theorem for self-dual complexes. Throughout Section \ref{Section_Supp}, until subsection \ref{spread_out}, we assume that the base field $\mathbf{k}$ is a finite field with $\overline{\mathbf{k}}$ its algebraic closure. We assume that $l$ is a prime number coprime to the characteristic of $\mathbf{k}$ when we work with $l$-adic sheaves. For notational convenience, we omit Tate twists when it does not cause confusion.

Let $B$ be a scheme over $\mathbf{k}$. Let $g: P \to B$ be a smooth $B$-group scheme with geometrically connected fibers, and let $f: M \to B$ be a proper morphism with $M$ quasi-projective. Assume that the group scheme $P$ acts on $M$ via
\begin{equation}\label{action}
a: P\times_BM  \to M.
\end{equation}
We say that the triple $(M, P, B)$ is a \emph{weak abelian fibration} of relative dimension $d$, if 
\begin{enumerate}
    \item[(i)] every fiber of the map $g$ is pure of dimension $d$, and $M$ has pure dimension 
    \begin{equation}\label{1.1_2}
    \mathrm{dim}M =d  +\mathrm{dim}B,
    \end{equation} 
    \item[(ii)] the action (\ref{action}) of $P$ on $M$ has \emph{affine} stabilizers, and
    \item[(iii)] the Tate module $T_{\overline{\mathbb{Q}}_l}(P)$ associated with the group scheme $P$ is polarizable.
\end{enumerate}
The notion of weak abelian fibration was introduced by Ng\^o \cite{Ngo} modelled on Hitchin's integrable systems \cite{Hit, Hit1}. We refer to Section \ref{Section1.3} for a brief review of Tate modules and their polarizations.

For a closed point $s \in B$, we denote by $\delta(s)$ the dimension of the affine part of the algebraic group $P_s$. This defines an upper semi-continuous function
\[
\delta: B \to \BN, \quad s \mapsto \delta(s). 
\]
For a closed subvariety $Z \subset B$, we define $\delta(Z)$ to be the minimal value of the function $\delta$ on $Z$.

The following is our main theorem.

\begin{thm}\label{supp_thm}
Let $(M, P, B)$ be a weak abelian fibration of relative dimension $d$. Let $\CK \in D^b_c(M, \overline{\BQ}_l)$ be a $P$-equivariant bounded complex satisfying the following properties:
\begin{enumerate}
    \item[(a)](Decomposition Theorem) The direct image complex $Rf_* \CK$ admits a (non-canonical) decomposition 
    \begin{equation}\label{support}
    Rf_* \CK \simeq \bigoplus_i {^\mathfrak{p}\CH}^i(Rf_* \CK)[-i].
    \end{equation}
    Moreover, after a base change to  $B_{\overline{\mathbf{k}}}=B\times_{\mathbf{k}}\overline{\mathbf{k}}$, the perverse sheaves ${^\mathfrak{p}\CH}^i(Rf_* \CK)$ are semisimple of the form  
    \[
{^\mathfrak{p}\CH}^i(Rf_* \CK) = \bigoplus_{\alpha} \mathrm{IC}_{Z_{\alpha,i}}(L_{\alpha,i})
    \]
where $Z_{\alpha,i}$ is a closed irreducible subvariety of $B_{\overline{\mathbf{k}}}$ and each $L_{\alpha,i}$ is a pure simple local system of weight $i$ on an open dense subset of $Z_{\alpha,i}$. We call these $Z_{\alpha,i}$ the supports of the decomposition (\ref{support});
    \item[(b)](Duality) We have an isomorphism 
    \[
    \BD(\CK) \simeq \CK [2 \mathrm{dim}M]
    \]
    with $\BD(-)$ the dualizing functor on $M$;
    \item[(c)] (Relative Dimension Bound) For the standard truncation functor $\tau_{> *}(-)$, we have
    \[
    \tau_{> 2d} \left( Rf_* \CK \right) =0.
    \]
\end{enumerate}
Then for any support $Z$ of the decomposition (\ref{support}), we have the inequality 
\begin{equation}\label{supp_dim}
\mathrm{codim}Z \leq \delta_{Z}.
\end{equation}
\end{thm}

In \cite{Ngo}, Ng\^o worked with the trivial local system $\overline{\mathbb{Q}}_l$ on $M$ where he assumed that the conditions (a) and (b) hold. Furthermore, he assumed that every fiber of $f$ is pure of dimension $d$, where the condition (c) follows automatically by the base change. Therefore, Theorem \ref{supp_thm} is a generalization of Ng\^o's support theorem \cite[Theorem 7.2.1 and Proposition 7.2.2]{Ngo}. We note that (c) is a crucial condition for the support theorem to hold for general $\CK$ as in Theorem \ref{supp_thm}. We first illustrate this in the following special case of Theorem \ref{supp_thm} --- the Goresky--MacPherson inequality.

\subsection{The Goresky--MacPherson inequality}

If the group scheme $P$ is affine whose action on $M$ is trivial, Theorem \ref{supp_thm} then specializes to the following theorem, which is known as the \emph{Goresky--MacPherson inequality} when $M$ is nonsingular and $\CK = \overline{\BQ}_l$.

\begin{thm}\label{thm1.2}
Let $f: M \to B$ be a proper map with $\mathrm{dim}M = \mathrm{dim}B +d$. Assume $\CK \in D^b(M, \overline{\BQ}_l)$ satisfies (a,b,c) of Theorem \ref{supp_thm}. Then any support $Z$ of (\ref{support}) satisfies the inequality
\[
\mathrm{codim}(Z) \leq d.
\]
\end{thm}

We first provide a proof of Theorem \ref{thm1.2} since it contains the main ingredients in the proof of Theorem \ref{supp_thm}, and in particular demonstrates the role played by the conditions (a,b,c).

\begin{proof}
Let $Z$ be a support. We denote
\begin{align*}
    \mathrm{occ}(Z)& :=\{i\in \BZ, ~~~ {^\mathfrak{p}\CH}^i(Rf_* \CK) \textup{ contains a simple factor with support } Z\}, \\
    \mathrm{amp}(Z) & := \mathrm{max}(\mathrm{occ}(Z))- \mathrm{min}(\mathrm{occ}(Z)).
\end{align*}

By (b), the set $\mathrm{occ}(Z)$ is symmetric with respect to the integer $\mathrm{dim}M$. This allows us to pick $m \in \mathrm{occ}(Z)$ with $m \geq \mathrm{dim}M$. In particular, we have ${^\mathfrak{p}\CH}^m(Rf_* \CK) \neq 0$. Hence by (a) there exists an open subset $U \subset Z$ and a local system $\CL$ on $U$ such that the shifted perverse sheaf
\[
\big{(}\CL [\mathrm{dimZ}]\big{)} [-m]= \CL[\mathrm{dim}Z -m]
\]
is a direct sum component of the complex $(Rf_*\CK)|_U$. We obtain that
\begin{equation}\label{Hp}
\CH^{m-\mathrm{dim}Z}(Rf_*\CK) \neq 0 \in D^b_c(B, \overline{\BQ}_l).
\end{equation}
By (\ref{Hp}) and the condition (c), we conclude that
\[
\mathrm{dim}M - \mathrm{dim}Z \leq m- \mathrm{dim}Z \leq 2d
\]
where the first inequality follows from the choice of $m$. This completes the proof of Theorem \ref{thm1.2} thanks to (\ref{1.1_2}).
\end{proof}

As observed by Ng\^o \cite[Proposition 7.3.2]{Ngo}, for a weak abelian fibration $(M,P,B)$ and an object $\CK$ as in Theorem \ref{supp_thm}, if we have 
\begin{equation}\label{amp_ineq}
    \mathrm{amp}(Z) \geq 2(d-\delta_Z),
\end{equation}
then the integer $m$ in the proof of Theorem \ref{thm1.2} can be chosen such that 
\[
m \geq  \mathrm{dim}M + (d- \delta_Z).
\]
An identical argument as above implies (\ref{supp_dim}).

In conclusion, the following proposition implies Theorem \ref{supp_thm}.

\begin{prop}\label{Prop1.3}
Under the assumption of Theorem \ref{supp_thm}, the inequality (\ref{amp_ineq}) holds for any support $Z$ of $Rf_*\CK$.
\end{prop}

The rest of Section \ref{Section_Supp} is devoted to proving Proposition \ref{Prop1.3}. We will further reduce Proposition \ref{Prop1.3} to a fiberwise ``freeness" statement as stated in Proposition \ref{Prop1.5}. Essentially, the arguments of \cite[Section 7]{Ngo} can be modified to prove this more generalized version of ``freeness" under our assumptions. We point out the necessary modifications (see Propositions \ref{prop1.4} and \ref{prop1.6}) and sketch all major steps in the proof following \cite[Section 7]{Ngo} for the reader's convenience.

\subsection{Actions of the group scheme}\label{Section1.3}
As part of the data of a weak abelian fibration $(M,P,B)$, the group scheme $g: P \to B$ is smooth over $B$ with $d$-dimensional geometrically connected fibers, which defines a complex
\[
\Lambda_P :=Rg_! \overline{\BQ}_l[2d]
\]
on the base $B$. The stalk of each cohomology sheaf $\CH^{-i}(\Lambda_P)$ over a closed point $s \in B$ computes the $i$-th homology of the group $P_s$,
\[
\CH^{-i}(\Lambda_P)_s = H^{2d-i}_c(P_s, \overline{\BQ}_l) = H_i(P_s, \overline{\BQ}_l).
\]
The Tate module associated with $g: P\to B$ is defined to be
\begin{equation}\label{Tate_mod}
T_{\overline{\BQ}_l}(P) := \CH^{-1}(\Lambda_P).
\end{equation}
Note that the complex $\Lambda_{(-)}$ and the sheaf $T_{\overline{\BQ}_l}(-)$ are defined for any smooth group scheme over $B$.
In our setting, as shown in \cite[Section 7.4.3]{Ngo}, the
group structure $\mu: P \times_B P \rightarrow P$
induces a convolution product
\[
\tau:\Lambda_P \otimes \Lambda_P \rightarrow \Lambda_P.
\]
Furthermore, it is also shown there that equation \eqref{Tate_mod} extends to a natural isomorphism
\[
\Lambda_P = \bigoplus \bigwedge^i T_{\overline{\BQ}_l}(P)[i]
\]
compatible with the multiplication action on each side.

We consider the Chevalley decomposition of the nonsingular commutative group $P_s$
\begin{equation}\label{Chevalley}
1 \to R_s \to P_s \to A_s \to 1
\end{equation}
over any geometric point $s\in B$ where $R_s$ is affine and $A_s$ is an abelian variety. This induces the short exact sequence of Tate modules
\begin{equation}\label{Tate_chevalley}
0 \to T_{\overline{\BQ}_l}(R_s) \to T_{\overline{\BQ}_l}(P_s) \to T_{\overline{\BQ}_l}(A_s) \to 0.
\end{equation}
Following \cite[Section 7.1.4]{Ngo}, we say that the Tate module (\ref{Tate_mod}) is polarizable if \'etale locally there exists a bilinear form 
\[
T_{\overline{\BQ}_l}(P) \times T_{\overline{\BQ}_l}(P) \to \overline{\BQ}_l
\]
which induces a non-degenerate pairing on $T_{\overline{\BQ}_l}(A_s)$ for any $s\in B$ via the quotient map of (\ref{Tate_chevalley}).

The following proposition generalizes the cap product action constructed in \cite[Section 7.4.2]{Ngo}.

\begin{prop}\label{prop1.4}
Let $(M,P,B)$ be a weak abelian fibration of relative dimension $d$, and let $\CK \in D_c^b(M, \overline{\BQ}_l)$ be a $P$-equivariant object. Then the $P$-action (\ref{action}) on $M$ induces an action of $\Lambda_P$ on $Rf_* \CK$,
\begin{equation}\label{prop1.4_1}
c: \Lambda_P \otimes Rf_* \CK \to Rf_* \CK.
\end{equation}
Furthermore, the compositions $c \circ (\tau \otimes \mathrm{id})$ and $c \circ (\mathrm{id}\otimes c)$
define the same morphism
\[ \Lambda_P \otimes \Lambda_P \otimes Rf_* \CK \to Rf_* \CK.
\]
\end{prop}

\begin{proof}
The trace map 
\[
Ra_! \overline{\BQ}_l [2d] \to \overline{\BQ}_l
\]
on $M$ associated with (\ref{action}) induces a morphism
\begin{equation}\label{trace1}
Ra_! \overline{\BQ}_l [2d] \otimes \CK \to \CK.
\end{equation}
We consider the Cartesian diagram 
\begin{equation}\label{diagram1}
\begin{tikzcd}
P\times_BM \arrow[r, "p_M"] \arrow[d, "p_P"]
& M \arrow[d, "f"] \\
P \arrow[r, "g"]
& B
\end{tikzcd}
\end{equation}
with $p_M$ and $p_P$ the projections. The lefthand side of (\ref{trace1}) is equal to
\begin{align*}
    Ra_!\left( \overline{\BQ}_l[2d] \otimes a^* \CK\right) =Ra_! \left( \overline{\BQ}_l[2d]\otimes p_M^*\CK \right).
\end{align*}
Here we used the projection formula and the isomorphism $\phi: a^*\CK \simeq p_M^*\CK$ given by the $P$-equivariance of $\CK$. Hence we obtain the morphism 
\[
Ra_! \left( \overline{\BQ}_l[2d]\otimes p_M^*\CK \right) \to \CK.
\]
Applying the functor $Rf_*$ to the morphism above and noticing that $Rf_* = Rf_!$, we have
\[
Rf_!Ra_! \left( \overline{\BQ}_l[2d]\otimes p_M^*\CK \right) \to Rf_*\CK
\]
where the lefthand side can be rewritten as
\begin{align*}
    Rf_!Ra_! \left( \overline{\BQ}_l[2d]\otimes p_M^*\CK \right) & =  Rf_!Rp_{M!} \left( \overline{\BQ}_l[2d]\otimes p_M^*\CK \right)\\
    &=  Rf_!\left(Rp_{M!} \overline{\BQ}_l[2d]\otimes \CK \right)\\
    & = Rf_!\left(Rp_{M!}(p_P^* \overline{\BQ}_l[2d])\otimes \CK \right)\\
    & = Rf_!\left(f^* Rg_! \overline{\BQ}_l[2d]\otimes \CK \right) \\
    & = Rg_!\overline{\BQ}_l[2d] \otimes Rf_!\CK
\end{align*}
where the first equality follows from 
\[
f\circ a = g\times_Bf: P\times_BM \to B,
\]
the second equality is given by the projection formula, the third equality follows from $p_P^*\overline{\BQ}_l = \overline{\BQ}_l$, the fourth equality is the base change 
\[
{Rp_M}_! p_P^* = f^*Rg_!
\]
with respect to the diagram (\ref{diagram1}), and the last equality is again given by the projection formula. 

To show the second claim of the proposition, we apply the same construction from above to the commutative diagram
\begin{equation}\label{diagram-action}
\begin{tikzcd}
P\times_B P\times_B M \arrow[r, "\mathrm{id}_P\times a"] \arrow[d, "\mu\times \mathrm{id}_M"]
& P \times_B M \arrow[d, "a"] \\
P \times_B M \arrow[r, "a"]
& M.
\end{tikzcd}
\end{equation}
Again using the trace map, each path defines a morphism
\[
\Lambda_P \otimes \Lambda_P \otimes Rf_* \CK \to Rf_* \CK
\]
The path via the lower-left corner gives the morphism $c \circ (\tau \otimes \mathrm{id})$ and the path via the upper-right corner gives the morphism $c \circ (\mathrm{id}\otimes c)$.  The equivariant structure on $\CK$ implies a cocycle condition on the isomorphism $\phi$ after pullback to $P\times_B P\times_B M$; this cocycle condition implies that these two morphisms agree.

This completes the proof of Proposition \ref{prop1.4}.
\end{proof}

\subsection{Actions on each support and freeness}\label{Sec1.4}
We denote by $I$ the set of the supports $Z \subset B_{\overline{\mathbf{k}}}$ of $Rf_* \CK$. In view of the condition (a) of Theorem \ref{supp_thm}, we have a canonical decomposition of the perverse sheaf ${^\mathfrak{p}\CH}^i(Rf_* \CK)$ in terms of the supports
\[
{^\mathfrak{p}\CH}^i(Rf_* \CK) = \bigoplus_{\alpha\in I} \CK^i_\alpha
\]
where $\CK^i_\alpha$ has support $Z_\alpha$ indexed by $\alpha \in I$. We collect all the direct summands of $Rf_*\CK$ with support $Z_\alpha$,
\begin{equation}\label{K_alpha}
\CK_\alpha : = \bigoplus_{i} \CK^i_{\alpha}.
\end{equation}

In the following 4 steps, we prove that Proposition \ref{Prop1.3} can be reduced to a freeness property concerning stalks of (\ref{K_alpha}).

\noindent {\bf Step 1.} For any support $Z_\alpha$ of $Rf_*\CK$, we may find an open dense subset $V_\alpha \subset Z_\alpha$ such that:
\begin{enumerate}
    \item[(i)] the restriction of $\CK^i_\alpha$ to $V_\alpha$ is of the form $\CL^i_\alpha [\mathrm{dim}V_\alpha]$ with $\CL^i_\alpha$ a pure local system of weight $i$;
    \item[(ii)] the restriction $P_\alpha$ of the group scheme $P$ to the support $Z_\alpha$ admits a Chevalley decomposition 
    \begin{equation}\label{step1_1}
    1 \to R_\alpha \to P_\alpha \to A_\alpha \to 1
    \end{equation}
    whose induced short exact sequence of Tate modules
    \[
    0 \to T_{\overline{\BQ}_l}(R_\alpha) \to T_{\overline{\BQ}_l}(P_\alpha) \to T_{\overline{\BQ}_l}(A_\alpha) \to 0
    \]
    satisfies that $T_{\overline{\BQ}_l}(R_\alpha)$ is a pure local system of weight -2; and
    \item[(iii)] for any other support $Z_{\alpha'}$, we have $Z_{\alpha'} \cap V_\alpha = \emptyset$ unless $Z_\alpha \subset Z_{\alpha'}$.
\end{enumerate}

Since Step 1 (i, iii) are standard and (ii) only concerns the group scheme $P$, this follows identically from \cite[Section 7.4.8]{Ngo}. \medskip

\noindent {\bf Step 2.} In \cite[Section 7.4.6]{Ngo}, we replace $Rf_! \overline{\BQ}_l$ by $Rf_! \CK$, and we replace the cap product action
\[
\Lambda_P \times Rf_! \overline{\BQ}_l \rightarrow Rf_! \overline{\BQ}_l
\]
of \cite[Section 7.4.2]{Ngo} by the action (\ref{prop1.4_1}) constructed in Proposition \ref{prop1.4}:
\begin{equation*}\label{step2_0}
\Lambda_P \otimes Rf_! \CK \to Rf_!\CK. 
\end{equation*}
As a consequence, for each $\alpha \in I$ and $i$ we obtain an morphism
\begin{equation}\label{step2_1}
T_{\overline{\BQ}_l}(P_\alpha) \otimes \CK^i_\alpha \to \CK^{i-1}_{\alpha}.
\end{equation}
The last statement follows from an identical argument as in \cite[Sections 7.4.6 and 7.4.7]{Ngo}. More precisely, perverse truncation functors yield 
\[
T_{\overline{\BQ}_l}(P_\alpha) \otimes {^\mathfrak{p}\CH}^i(Rf_! \CK) \to {^\mathfrak{p}\CH}^{i-1}(Rf_! \CK),
\]
which can be further written as
\[
\bigoplus_{\alpha\in I} T_{\overline{\BQ}_l}(P_\alpha) \otimes \CK^i_\alpha \to \bigoplus_{\alpha\in I} \CK^{i-1}_\alpha 
\]
in terms of the supports $Z_\alpha$. This gives the canonical morphism (\ref{step2_1}).

\medskip
\noindent {\bf Step 3.} Now we combine Steps 1 and 2. Consider the restriction of $\CK_\alpha$ to $V_\alpha$ of Step 1,
\[
\CL_\alpha : = \bigoplus_{i} \CL^i_\alpha[-i].
\]

Using the last part of Proposition \ref{prop1.4},
the morphisms
(\ref{step2_1})
extend to an action of the local system of graded algebras
$\Lambda_{P_\alpha} = \bigoplus \bigwedge^i T_{\overline{\BQ}_l}(P_\alpha) [i]$
on $\CL_\alpha$.

As explained in \cite[Section 7.4.9]{Ngo} the first paragraph in Page 121,
an argument using weights shows that (\ref{step2_1}) passes through an action of the abelian variety part $T_{\overline{\BQ}_l}(A_\alpha)$ of the Tate module $T_{\overline{\BQ}_l}(P_\alpha)$.
 As a result, we have a graded module structure on $\CL_\alpha$ of the graded algebra $\Lambda_{A_\alpha}$ associated with the abelian scheme $A_\alpha$ in (\ref{step1_1}),
\begin{equation}\label{step3_1}
    \Lambda_{A_\alpha} \otimes \CL_\alpha \to \CL_\alpha.
\end{equation}

Note that we use the assumption that $\mathbf{k}$ is a finite field here.
\medskip

\noindent {\bf Step 4.} As commented in the paragraph after \cite[Proposition 7.4.10]{Ngo}, Proposition \ref{Prop1.3} can be deduced from the following proposition.

\begin{prop}\label{Prop1.5}
We follow the same notation as in Steps 1-3 above. Let $u_\alpha \in V_\alpha$ be any geometric point. Then the stalk $\CL_{\alpha, u_\alpha}$ of $\CL_\alpha$ is a free graded module of the graded algebra $\Lambda_{A_\alpha, u_\alpha}$ under the action (\ref{step3_1}). 
\end{prop}

We complete the proof of Proposition \ref{Prop1.5} in the next two sections.

\subsection{Freeness} In this section we prove the following proposition which generalizes \cite[Proposition 7.5.1]{Ngo}. Then in Section \ref{sec1.6} we eventually reduce Proposition \ref{Prop1.5} to Proposition \ref{prop1.6}.

\begin{prop}\label{prop1.6}
Assume $X$ is projective over $\mathbf{\overline{k}}$ which admits an action of an abelian variety $A$ over $\mathbf{\overline{k}}$ with finite stabilizers. Let $\CE \in D_c^b(X, \overline{\BQ}_l)$ be an $A$-equivariant object. Then the graded cohomology group
    \begin{equation}\label{prop6.1_1}
    \bigoplus_{i}H^i(X, \CE)[-i]
    \end{equation}
    is naturally a free graded module of the graded algebra $\Lambda_A=\oplus_iH^i(A,\overline{\BQ}_l)[-i]$.
\end{prop}

\begin{proof}
Since the $A$-action preserve the connected components of $X$, we may assume that $X$ is connected. We consider the quotient map
\[
q: X \to Y := X/A
\]
with $X/A$ an Artin stack with finite inertia. Thanks to the projectivity of $A$, the morphism $q$ is smooth and proper. For the $A$-equivariant object $\CE$, there exists an object $\CE'$ on $Y$ such that 
\[
q^* \CE' = \CE,
\]
and the projection formula yields
\begin{equation}\label{RfE}
Rq_* \CE = Rq_* \overline{\BQ}_l \otimes \CE'.
\end{equation}
In particular, the complex $Rq_*\CE$ admits a natural $\Lambda_A$-action through the first factor of the righthand side of (\ref{RfE}). This shows that (\ref{prop6.1_1}) is a natural graded $\Lambda_A$-module.

Now since $q$ is smooth and proper, we have a decomposition\footnote{As explained in the proof of \cite[Proposition 7.5.1]{Ngo}, the decomposition here is induced by the cup-product with an relative ample class. We also refer to \cite{Sun} as a general reference for the decomposition theorem for Artin stacks with affine stabilizers.}
\begin{equation}\label{prop1.6_2}
Rq_* \overline{\BQ}_l \simeq \bigoplus_i R^iq_*\overline{\BQ}_l[-i].
\end{equation}
Moreover, we consider the Cartesian diagram 
\[
\begin{tikzcd}
A\times X \arrow[r,"q'"] \arrow[d, "q'"]
& X \arrow[d, "q"] \\
X \arrow[r, "q"]
& Y.
\end{tikzcd}
\]
with all the arrows smooth maps. By the base change, we obtain the canonical $A$-equivariant isomorphism of local systems 
\begin{equation}\label{323232}
q^*R^iq_*\overline{\BQ}_l =R^i{q'}_*( {q'}^*\overline{\BQ}_l).
\end{equation}
We have that ${q'}^*\overline{\BQ}_l$ is a trivial local system of rank 1 and an $A$-equivariant structure on it is trivial by the connectedness of $A$ (\emph{c.f.} \cite[Lemma A.1.2]{Zhu}). In particular $q^*R^iq_*\overline{\BQ}_l$ is a trivial local system equipped with the trivial $A$-equivariant structure by (\ref{323232}). Consequently, each $R^iq_*\overline{\BQ}_l$ is canonically isomorphic to the trivial local system taking value in $H^i(A, \overline{\BQ}_l)$. The filtration $(\tau_{\leq *}Rq_* \overline{\BQ}_l) \otimes \CE'$ of $Rq_*\CE$ induces the following spectral sequence
\[
H^j(Y, R^iq_*\overline{\BQ}_l \otimes \CE') = H^j(Y, \CE')\otimes H^i(A, \overline{\BQ}_l) \Rightarrow  H^{i+j}(X, \CE),
\]
which degenerates thanks to (\ref{RfE}) and the decomposition (\ref{prop1.6_2}). Hence we obtain a filtration stable under the $\Lambda_A$-action whose graded pieces are the free graded $\Lambda_A$ modules
\[
H^j(Y, \CE')\otimes \left( \bigoplus_i H^i(A, \overline{\BQ}_l)\right).
\]
This proves the freeness of the entire module $H^*(X, \CE)=\bigoplus_{i}H^i(X, \CE)[-i]$.
\end{proof}

\subsection{Proof of Proposition \ref{Prop1.5}}\label{sec1.6}
We deduce Proposition \ref{Prop1.5} from Proposition \ref{prop1.6} by a descending induction on the dimension of the support $Z_\alpha$. This is parallel to \cite[Section 7.7]{Ngo}.

We complete the induction in the following 3 steps.
\medskip

\noindent {\bf Step A.} The induction base follows from Proposition \ref{prop1.6} which we explain as follows. We assume $Z_{\alpha_0} = B_{\overline{\bf k}}$ and $V_{\alpha_0}$ is an open dense subset of $Z_{\alpha_0}$ as in Step 1 of Section \ref{Sec1.4}. All the other $Z_\alpha$ with $\alpha \neq \alpha_0$ does not intersect with $V_{\alpha_0}$, and 
\[
{^\mathfrak{p}\CH}^i(Rf_* \CK)|_{V_{\alpha_0}} = \CL^i_{\alpha_0}[\mathrm{dim} B].
\]
Therefore, for any geometric point $u_{\alpha_0}$ of $V_{\alpha_0}$ with 
\[
\iota_{u_{\alpha_0}}:M_{u_{\alpha_0}} \hookrightarrow M
\]
the corresponding fiber, we have the identification
\begin{equation}\label{StepA_1}
  \bigoplus_i \CL^i_{\alpha_0,u_{\alpha_0}}[-i+\mathrm{dim}B] = \bigoplus_i H^i\left(M_{u_{\alpha_0}}, \iota_{u_{\alpha_0}}^*\CK\right)[-i] 
\end{equation}
by the base change. Parallel to Step 3 of Section \ref{Sec1.4}, (\ref{StepA_1}) admits a natural $\Lambda_{A_{{\alpha_0},u_{\alpha_0}}}$-action induced by the action of $\Lambda_{P_{\alpha_0},u_{\alpha_0}}$.

As explained in the last paragraph of \cite[Section 7.7.1]{Ngo}, we may assume that the geometric point $u_{\alpha_0}$ is defined over a finite field. So there exists a quasi-lifting $A_{\alpha_0, u_{\alpha_0}} \to P_{\alpha_0, u_{\alpha_0}}$ (see \cite[Proposition 7.5.3]{Ngo}) such that the $\Lambda_{A_{\alpha_0}, u_{\alpha_0}}$-action on (\ref{StepA_1}) is induced by the $A_{{\alpha_0}, u_{\alpha_0}}$-action on $M_{u_{\alpha_0}}$. By the axiom (ii) of weak abelian fibrations, the $A_{{\alpha_0}, u_{\alpha_0}}$-action on $M_{u_{\alpha_0}}$ passing through $P_{\alpha_0, u_{\alpha_0}}$ has finite stabilizers. Hence Proposition \ref{prop1.6} implies that (\ref{StepA_1}) is free over $\Lambda_{A_{\alpha_0}, u_{\alpha_0}}$. This completes the proof of the induction base.
\medskip

\noindent {\bf Step B.} Since Proposition \ref{Prop1.5} is a local statement, we may work with a strictly Henselian base. Assume that $B_\alpha$ is the strict Henselization of a geometric point $u_\alpha$ defined over a finite field lying in $V_{\alpha} \subset Z_\alpha$. By the choice of $V_\alpha$ in Step 1 of Section \ref{Sec1.4}, the stalk $\CK^i_{\alpha',u_\alpha}$ is non-zero only if $Z_\alpha$ is strictly contained in $Z_{\alpha'}$. In this case, the induction assumption implies that, for any $m \in \BZ$, the graded $\overline{\BQ}_l$-vector space
\[
\bigoplus_{i}H^m\left( \CK^i_{\alpha',u_\alpha} \right)[-i]
\]
is equipped with a natural free $\Lambda_{A_\alpha,{u_\alpha}}$-action induced by (\ref{prop1.4_1}). This is explained in \cite[Proposition 7.7.4]{Ngo} which essentially relies on the polarizability of $P$, \emph{i.e.} the axiom (iii) of weak abelian fibrations.\footnote{Since this part only concerns the group scheme $P$, the proof of \cite[Proposition 7.7.4]{Ngo} applies identically here.}
\medskip

\noindent {\bf Step C.} We complete the induction argument.

The condition (a) of Theorem \ref{supp_thm} (\emph{i.e.}, the decomposition theorem for $Rf_*\CK$) implies the degeneracy of the spectral sequence
\begin{equation}\label{ss1}
H^j({^\mathfrak{p}\CH}^i(Rf_* \CK)_{u_\alpha}) \Rightarrow H^{i+j}\left(M_{u_\alpha}, \iota_{u_\alpha}^* \CK\right)
\end{equation}
where $\iota_{u_\alpha}: M_{u_\alpha} \hookrightarrow M$ is the geometric fiber over $u_\alpha$. This induces a $\Lambda_{A_{\alpha,{u_\alpha}}}$-stable filtration $F^\bullet \BH$ on the total cohomology
\[
\BH : = \bigoplus_{i} H^{i}\left(M_{u_\alpha}, \iota_{u_\alpha}^* \CK\right)[-i]
\]
whose $m$-th graded piece is
\begin{equation}\label{gradedpiece}
F^m\BH/F^{m+1}\BH = \bigoplus_{i} H^m\left(  {^\mathfrak{p}\CH}^i(Rf_* \CK)_{u_\alpha}\right) [-i-m].
\end{equation}
In addition, we have the following:
\begin{enumerate}
\item[(1)] By picking a quasi-lifting $A_{\alpha,u_{\alpha}} \rightarrow P_{\alpha, u_\alpha}$ as in the last paragraph of Step A, it follows from Proposition \ref{prop1.6} and the $P$-equivariance of $\CK$ that $\BH$ is a free graded $\Lambda_{A_{\alpha, u_\alpha}}$-module.
    \item[(2)] Since 
    \[
   \bigoplus_i {^\mathfrak{p}\CH}^i(Rf_* \CK)\big{|}_{B_\alpha} [-i]= \bigoplus_{\alpha'}\bigoplus_{i}\CK^i_{\alpha',u_{\alpha}}[-i],
    \]
    the graded piece (\ref{gradedpiece}), as a graded $\Lambda_{A_{\alpha, u_\alpha}}$-module, is a direct sum of the graded $\Lambda_{A_{\alpha, u_\alpha}}$-modules 
    \[
     \bigoplus_{i}H^m\left(  \CK^i_{\alpha',u_\alpha}\right)[-i-m]
    \]
    over all $\alpha'$ with $Z_\alpha \subset Z_{\alpha'}$
    \item[(3)] By Step B, the induction assumption implies that each
    \[
   \bigoplus_{i} H^m\left(  \CK^i_{\alpha',u_\alpha} \right)[-i]
    \]
    is a free graded $\Lambda_{A_{\alpha, u_\alpha}}$-module when $Z_\alpha \subset Z_{\alpha'}$.
    \item[(4)] The graded $\Lambda_{A_{\alpha},u_\alpha}$-module vanishes:
    \[
    \bigoplus_i H^m\left( \CK^i_{\alpha, u_\alpha} \right)[-i]= \bigoplus_i H^{m+\mathrm{dim}V_\alpha}\left( \CL^i_{\alpha,u_\alpha}\right)[-i] =0
    \]
    if $m\neq -\mathrm{dim}V_\alpha$ for degree reasons, since $\CL^i_{\alpha,u_\alpha}$ is a sky-scraper sheaf supported at $u_\alpha$.
\end{enumerate}
Recall the filtration $F^\bullet\BH$ associated with the spectral sequence (\ref{ss1}) whose grades pieces are given by (\ref{gradedpiece}). We arrive at exactly the situation of the last paragraph in \cite[Page 131]{Ngo}: the spectral sequence (\ref{ss1}) induces a 3-layer filtration of $\Lambda_{A_\alpha, u_\alpha}$-modules
\[
0 \subseteq F^{n+1}\BH \subseteq F^n\BH \subseteq \BH, \quad \quad n = -\mathrm{dim}V_\alpha = -\mathrm{dim}Z_\alpha,
\]
where
\begin{enumerate}
    \item[$\bullet$] $\BH$ is free by (1), and
    \item[$\bullet$] $F^{n+1}\BH$ and $\BH/{F^n\BH}$ are free by (3) and (4). In fact, (4) ensures that  \[
    \bigoplus_i H^m\left( \CK^i_{\alpha, u_\alpha} \right)[-i]
    \]
    vanishes when $m \neq n$, and therefore $F^m\BH/F^{m+1}\BH$ is free by (3).
\end{enumerate}
This implies the freeness for 
\[
F^n\BH/{F^{n+1}\BH} = \left( \bigoplus_i \CL^i_{\alpha,u_\alpha}[-i-n]\right) \oplus \left( \bigoplus_{\alpha' \neq \alpha} \bigoplus_i H^n\left(  \CK^i_{\alpha',u_\alpha}\right)[-i-n] \right)
\]
which completes the induction; see \cite[Page 131-132]{Ngo}.  
\qed

\begin{rmk}
We fixed some minor typos in \cite[Section 7.7.2]{Ngo}: the correct formula for the $m$-th graded piece \cite[Page 131, Line 9]{Ngo} of the Leray spectral sequence is 
\[
H^m\left( \bigoplus_n{^\mathfrak{p}\CH}^n(Rf_* \overline{\BQ}_l)_{s_0} \right)[-n-m],
\]
which is not equal to 
\[
H^m\left( \bigoplus_n{^\mathfrak{p}\CH}^n(Rf_* \overline{\BQ}_l[-n])_{s_0} \right)[-m]
\]
as stated in \cite{Ngo}. As consequences, the following statements in the last paragraph of \cite[Page 131]{Ngo} are incorrect:
\begin{enumerate}
    \item[$\bullet$] for $\alpha' \neq \alpha$, we have that $H^m(\bigoplus_{n\in \BZ} \CK^n_{\alpha',u_\alpha}[-n])$ is a free $\Lambda_{A_{u_\alpha}}$-module;
    \item[$\bullet$] For $\alpha = \alpha'$, we have that $H^m(\CK_{\alpha,u_\alpha})=0$ unless $m = - \mathrm{dim}Z_\alpha$;
\end{enumerate}
whose corrected versions are given in (2,3,4) of Step C above.
\end{rmk}

\subsection{Spread out for $\BC$}
\label{spread_out}

As a corollary of Theorem \ref{supp_thm}, we have the following theorem concerning the intersection cohomology complex of a weak abelian fibration $(M, P, B)$ over the complex numbers $\BC$.

\begin{thm}\label{thm1.8}
Let $(M,P,B)$ be a weak abelian fibration over $\BC$ of relative dimension $d$, \emph{i.e.}, the triple satisfies (i,ii,iii) of Section \ref{sec1.1}. Assume that 
\begin{equation}\label{bound}
\tau_{>2d}\left( Rf_* \mathrm{IC}_M[-\mathrm{dim}M]\right) =0.
\end{equation}
Then any support $Z$ of the decomposition for $Rf_*\mathrm{IC}_M$ satisfies the inequality
\[
\mathrm{codim}Z \leq \delta_Z.
\]
\end{thm}

\begin{proof}
By a standard spreading out argument (see for example \cite[Section 6]{BBD}), we may reduce Theorem \ref{thm1.8} to the same statement over finite fields. More precisely, we spread out the weak abelian fibration $(M, P, B)$ over $\mathrm{Spec}R$ where $R$ is a DVR of characteristic 0, such that the geometric fiber over a general prime $\mathfrak{p} \in \mathrm{Spec}R$ is a weak abelian fibration in characteritic $p$ as in the beginning of Section \ref{sec1.1}. Moreover, the condition (\ref{bound}) holds over a general prime in $\mathrm{Spec}R$. Therefore, if a support $Z$ of the decomposition theorem associated with $(M, P, B)$ violates the inequality $\mathrm{codim}Z \leq \delta_Z$, then by spreading out this inequality is violated by a support over a general prime $\mathfrak{p} \in \mathrm{Spec}R$ as well, which contradicts our assumption that Theorem \ref{thm1.8} holds over finite fields..

In our setting, note that the complex 
\[
\CK = \mathrm{IC}_M[-\mathrm{dim}M]
\]
is $P$-equivariant, which satisfies (a,b,c) of Theorem \ref{supp_thm} by the decomposition theorem, Verdier duality, and the condition (\ref{bound}) respectively. Hence we conclude Theorem \ref{thm1.8} from Theorem \ref{supp_thm}.
\end{proof}

In order to apply the support theorem to the intersection cohomology complex for a weak abelian fibration with \emph{singular} ambeint space $M$, the crucial point is to verify the ``relative dimension bound" (\ref{bound}). We discuss systematically in the next two sections how to obtain such a bound for the moduli of 1-dimensional sheaves and the moduli of semistable Higgs bundles.

\section{Moduli of 1-dimensional sheaves and Higgs bundles}\label{Sec2}

\subsection{Overview}
Throughout the rest of the paper, we work over the complex numbers $\BC$. We show in this section that the morphisms (\ref{0.2_1}) and (\ref{0.2_2}) admit the structures of weak abelian fibrations.

A crucial techinical result is Proposition \ref{prop2.5} concerning a dimension bound for certain moduli of pure 1-dimensinoal sheaves. As a consequence, we verify in Theorem \ref{thm2.3} the irreducibility of the moduli spaces $M^L_{\beta,\chi}$ of (A) which may be of independent interest. 

The dimension bound given by Proposition \ref{prop2.5} will be used again in Section 3 which plays an important role in the proof of our main theorems.

\subsection{Curves in del Pezzo surfaces}\label{Sec2.2}
Let $S$ be a del Pezzo surface, \emph{i.e.}, a nonsingular projective surface with $-K_S$ ample.


\begin{lem}\label{lem2.1}
Let $E$ be an effective divisor on $S$. Then
\[
\mathrm{dim}H^1(S, \CO_S(E)) = \mathrm{dim}H^2(S, \CO_S(E)) =0.
\]
In particular, we have $\mathrm{dim}H^0(S, \CO_S(E)) = \frac{1}{2}E\cdot(E-K_S)+1$.
\end{lem}

\begin{proof}
By Serre duality, we obtain
\[
H^2(S, \CO_S(E)) = H^0(S, \CO_S(K_S-E))^\vee
 =0.\]
Now we prove the vanishing of 
\begin{equation}\label{eqn2.2_1}
H^1(S, \CO_S(E))^\vee  = H^1(S, \CO_S(K_S-E)).
\end{equation}
We consider the short exact sequence
\[
0\to \CO_S(K_S-E) \to \CO_S(K_S) \to \CO_E(K_S) \to 0
\]
which induces the long exact sequence
\begin{equation}\label{eqn2.2_2}
\cdots \to H^0(S, \CO_E(K_S)) \to H^1(S, \CO_S(K_S-E)) \to H^1(S, \CO_S(K_S)) \to \cdots. 
\end{equation}
The vanishing $H^0(S, \CO_E(K_S)) = 0$ follows from $\mathrm{deg}_E(K_S)<0$, and Serre duality yields the vanishing $H^1(S, \CO_S(K_S))=0$. Hence (\ref{eqn2.2_2}) implies the vanishing of (\ref{eqn2.2_1}). 

The last statement follows from the Riemann--Roch formula.
\end{proof}

Let $\beta$ be an ample and effective class on $S$. Then Lemma \ref{lem2.1} implies that the base $B = \BP H^0(S, \CO_S(\beta))$ of (\ref{0.2_1}) is of dimension
\[
\mathrm{dim}B = \frac{1}{2}\beta\cdot(\beta- K_S).
\]
We define $\pi: \CC_B \to B$ to be the universal curve for the linear system $|\beta|$. Since $\beta$ is ample, it is base point free on the del Pezzo surface $S$. Hence Bertini theorem implies that a general member of $|\beta|$ is a nonsingular and integral curve of genus
\[
g_\beta = \frac{1}{2}\beta\cdot(\beta + K_S)+1.
\]
In particular, there exists a Zariski open dense subset $U \subset S$ such that the restriction of $\CC_B$ to $U$ is smooth 
\[
\pi: \CC \to U \subset B.
\]

We consider the relative degree 0 Picard variety
\[
P:= \mathrm{Pic}^0(\CC_B/B)
\]
parameterizing line bundles on the fibers of $\pi: \CC_B \to B$ whose restrictions to each irreducible component are of degree 0. The projection morphism
\[
\pi_P: P \to B
\]
has fibers of pure dimensions $g_\beta$. The restriction of $P$ to $U$ gives a smooth abelian scheme
\[
\pi_P: P_U\left(:=\mathrm{Pic}^0(\CC_U/U)\right) \to U.
\]
We also consider the relative degree $e$ Picard variety over $U$
\[
\pi_{P^e}: P^e_U\big{(}:=\mathrm{Pic}^e(\CC_U/U)\big{)}  \to U
\]
for any integer $e$. We recall the following well-known fact (see \cite[Lemma 1.3.5]{dCHM1}) concerning the variation of Hodge structures for Picard varieties of smooth curves.

\begin{prop}\label{prop2.2}
For any $e\in \BZ$, we have an isomorphism of variations of Hodge structures on $U$:
\[
R^i{\pi_{P^e}}_* \BQ_{P^e_U} \simeq \wedge^iR^1\pi_* \BQ_{\CC}.
\]
\end{prop}

\subsection{Moduli spaces of 1-dimensional sheaves}
Now we assume that $S$ is a \emph{toric} del Pezzo surface with a polarization $L$. The moduli space $M^L_{\beta,\chi}$ parameterizes $S$-equivalence classes of pure 1-dimensional (Gieseker-)semistable sheaves $\CF$ on $S$ with
\[
\mathrm{supp}(\CF) = \beta, \quad \chi(\CF)=\chi.
\]
Here the semistability is with respect of the slope function
\[
\mu(\CE) = \frac{\chi(\CE)}{c_1(\CE)\cdot L}.
\]

We recall the Hilbert--Chow morphism
\[
h: M^L_{\beta,\chi} \to B, \quad \CF \mapsto \mathrm{supp}(\CF)
\]
defined by taking the Fitting support \cite{LeP}. The open subvariety $h^{-1}(U) \subset M^L_{\beta,\chi}$ parameterizes line bundles supported on the nonsingular curves in $|\beta|$. Hence every fiber of $h$ over a closed point $b \in U$ is an abelian variety of dimenision $g_\beta$ and we have
\begin{equation}\label{eqn2.3_1}
h^{-1}(U) = \mathrm{Pic}^e(\CC_U/U), \quad e= \chi-1+g_\beta.
\end{equation}
The moduli space $M^L_{\beta,\chi}$ can be viewed as a compatification of the relative Picard variety (\ref{eqn2.3_1}).

The following theorem is of independent interest, and we postpone its proof to Section \ref{Sec2.6}.

\begin{thm}\label{thm2.3}
The moduli space $M^L_{\beta,\chi}$ is irreducible of dimension
\[
\mathrm{dim}M^L_{\beta,\chi} = \beta^2+1 = \mathrm{dim}B +g_\beta.
\]
\end{thm}

The group scheme $\pi_P: P \to B$ acts naturally and fiberwise on the moduli space $M^L_{\beta,\chi}$ via tensor product
\[
\CL \cdot \CF = \CL \otimes \CF, \quad \quad \CL \in P_b=\pi_P^{-1}(b),~~\CF \in h^{-1}(b)
\]
with $b \in B$ a closed point.

\begin{prop}\label{prop2.4}
The triple $\left(M^L_{\beta, \chi}, P, B\right)$ with $h$ and $\pi_P$ above form a weak abelian fibration of relative dimension $g_\beta$.
\end{prop}

\begin{proof}
We need to check (i,ii,iii) of Section \ref{sec1.1}. Recall from Section \ref{Sec2.2} that the group scheme $\pi_P: P \to B$ is smooth whose fibers are of pure dimensions $g_\beta$. Hence the condition (i) follows from Theorem \ref{thm2.3}. The affineness of the stabilizers (Condition (ii)) is proven in \cite[Lemma 3.5.4]{dCRS}, and the polarizability of the Tate module (Condition (iii)) associated with the group scheme $P$ is given by \cite[Theorem 3.3.1]{dC_SL} as explained in \cite[Lemma 3.5.5]{dCRS}.
\end{proof}

\subsection{Moduli stacks}\label{Sec2.4}
For the polarized surface $(S,L)$, the moduli of semistable sheaves can be constructed as a GIT-quotient of the corresponding Quot-scheme (denoted by $\mathrm{Quot}$),
\[
M^L_{\beta, \chi} = \mathrm{Quot}^{\mathrm{ss}} \sslash \mathrm{GL}_m
\]
where the semistable part of the Quot scheme $\mathrm{Quot}^{\mathrm{ss}}$ and $m$ rely on the Hilbert polynomial $\mathrm{dim}H^0(S, \CF\otimes L^{\otimes n})$ of a semistable sheaf $\CF$ with $\mathrm{supp(\CF)}=\beta$ and $\chi(\CF)=\chi$. We also consider the moduli stack of semistable sheaves 
\[
\CM^L_{\beta,\chi} = [\mathrm{Quot}^{\mathrm{ss}} \slash \mathrm{GL}_m]
\]
such that the natural projection 
\[
q: \CM^L_{\beta,\chi} \to M^L_{\beta, \chi}
\]
induces a \emph{good moduli space} of the Artin stack $\CM^L_{\beta,\chi}$. 

\begin{lem}\label{lem2.5}
The stack $\CM^L_{\beta,\chi}$ is nonsingular of dimension
\[
\mathrm{dim}\CM^L_{\beta,\chi} = \beta^2.
\]
\end{lem}

\begin{proof}
The obstruction space for a semistable sheaf $\CF \in \CM^L_{\beta,\chi}$ is
\begin{equation}\label{eqn_lem2.5}
\mathrm{Ext}_S^2(\CF, \CF) = \mathrm{Hom}_S(\CF, \CF\otimes \omega_S)^\vee, \quad \omega_S = \CO_S(K_S).
\end{equation}
We prove in the following that (\ref{eqn_lem2.5}) vanishes.

By the semicontinuity of $\mathrm{Hom}_S(\CF, \CF\otimes \omega_S)$, it suffices to show the vanishing
\[
\mathrm{Hom}_S(\CF, \CF \otimes \omega_S) =0
\]
when $\CF$ is a polystable sheaf on $S$. Hence we only need to prove the vanishing 
\begin{equation}\label{vanishing1}
\mathrm{Hom}_S(\CF_1, \CF_2\otimes \omega_S) = 0
\end{equation}
for two stable sheaves $\CF_1$ and $\CF_2$ on $S$ with the same slope
\[
\mu(\CF_1)= \mu(\CF_2).
\]
Since $-K_S$ is effective for the del Pezzo surface $S$, we have a short exact sequence
\[
0 \to \CF_2 \otimes \omega_S \to \CF_2 \to {\CF_2}|_E \to 0
\]
where $E$ is a curve in the linear system $|-K_S|$. The induced long exact sequence gives
\begin{equation}\label{LES}
0 \to \mathrm{Hom}_S(\CF_1, \CF_2\otimes \omega_S) \to \mathrm{Hom}_S(\CF_1, \CF_2) \to \mathrm{Hom}_S(\CF_1, \CF_2|_E). 
\end{equation}
When $\CF_1 \neq \CF_2$, by the stability we have $\mathrm{Hom}_S(\CF_1, \CF_2)=0$. When $\CF_1=\CF_2$, the second map of (\ref{LES}) is injective:
\[
\mathrm{Hom}_S(\CF_1, \CF_1)= \BC\cdot \mathrm{id} \hookrightarrow \mathrm{Hom}_S(\CF_1, \CF_1|_E).
\]
In particular (\ref{vanishing1}) vanishes in either case. This implies the vanishing of the obstuction (\ref{eqn_lem2.5}) and proves that $\CM^L_{\beta,\chi}$ is nonsingular. Consequently, we have
\[
\mathrm{dim}\CM^L_{\beta,\chi} = \mathrm{dim}\mathrm{Ext}_S^1(\CF, \CF) - \mathrm{dim}\mathrm{Hom}_S(\CF, \CF)= -\chi(\CF, \CF) = \beta^2. \qedhere
\]
\end{proof}

Combining with the Hilbert--Chow morphism $h: M^L_{\beta, \chi} \to B$, we obtain a morphism
\begin{equation}\label{stack_HC}
h_\CM: \CM^L_{\beta, \chi} \to B.
\end{equation}

\begin{prop}\label{prop2.5}
Let $S$ be a toric del Pezzo surface. For any closed point $b\in B$, we have the following dimension bound for the fiber of (\ref{stack_HC}): 
\begin{equation}\label{eqn_thm2.5}
\mathrm{dim} h_\CM^{-1}(b) \leq \frac{1}{2}\beta\cdot (\beta+K_S).
\end{equation}
\end{prop}

When $b$ represents an integral nonsingular curve, then $h_\CM^{-1}(b)$ is exactly a connected component of its Picard stack whose dimension $g_\beta-1$ matches the righthand side of (\ref{eqn_thm2.5}).

We prove Proposition \ref{prop2.5} in Section \ref{Sec2.5}. Then in Section \ref{Sec2.6} we use Proposition \ref{prop2.5} to complete the proof of Theorem \ref{thm2.3}.

\subsection{Proof of Proposition \ref{prop2.5}}\label{Sec2.5}
We reduce Proposition \ref{prop2.5} to a dimension bound for the nilpotent cone of Higgs bundles.

Let $C$ be a nonsingular curve of genus $g$, and let $D$ be a degree $d$ effective divisor on $C$ with 
\begin{equation}\label{d>2g-2}
d > 2g-2.
\end{equation}
We denote by $\CM^{\mathrm{nil}}_{n,\chi}$ the moduli stack of nilpotent Higgs bundles 
\[
(\CE, \theta): \quad \theta: \CE \to \CE \otimes \CO_C(D), \quad \mathrm{rank}(\CE)=n,\quad \chi (\CE) = \chi
\]
where we do not impose any (semi-)stability conditions.

The stack $\CM^{\mathrm{nil}}_{n,\chi}$ is essentially the central fiber of the (stacky) Hitchin fibration. Alternatively, by the spectral correspondence $\CM^{\mathrm{nil}}_{n,\chi}$ parameterizes pure 1-dimesional sheaves $\CF$ with
\[
\mathrm{supp}(\CF) = nC, \quad \chi(\CF) = \chi
\]
on the total space $\mathrm{Tot}(\CO_C(D))$ of the line bundle $\CO_C(D)$. Here the spectral correspondence is induced by the pushforward along the standard projection $\mathrm{Tot}(\CO_C(D)) \to C$.

\begin{prop}[\emph{c.f.}\cite{CL}]\label{prop2.6}
We have
\begin{equation}\label{eqn_thm2.6}
\mathrm{dim}\CM^{\mathrm{nil}}_{n,\chi} \leq n(g-1)+\frac{1}{2}n(n-1)d.
\end{equation}
\end{prop}

\begin{proof}
The dimension formula for the stack of the nilpotent cone and the comparison to the righthand side of (\ref{eqn_thm2.6}) are given in Line 2 and Line 6 at \cite[Page 725, Section 10]{CL}. Although it is assumed in the beginning of \cite{CL} that the curve $C$ has genus $g\geq 2$, the dimension calculation of \cite[Section 10]{CL} does not require this constraint as long as (\ref{d>2g-2}) holds.\footnote{An alternative proof of this dimension bound can be obtained using the method of \cite[Proposition 3.1]{Sch}. We note that the last equation of \cite[Section 10]{CL} shows that (\ref{eqn_thm2.6}) still holds if $d = 2g-2$.} 
\end{proof}

Now we prove Proposition \ref{prop2.5}.

We consider the maximal open torus $\BT \subset S$ whose action on $S$ induces $\BT$-actions on both the moduli stack $\CM^L_{\beta,\chi}$ and the base $B$. By a semi-continuity argument (\emph{c.f.} \cite[Proof of Corollary 1]{Ginzburg}), it suffices to show (\ref{eqn_thm2.5}) for all $T$-fixed points $b\in B$. Since we only concern dimension counts, we prove the following stronger statement for toric divisors without imposing (semi-)stability conditions.
\medskip

{\bf Claim:} For an effective divisor
\[
E= \sum_i n_i E_i, \quad n_i>0
\]
with each $E_i$ a nonsingular irreducible toric divisor, we have
\begin{equation}\label{eqn35}
\mathrm{dim}\CM_{E, \chi} \leq \frac{1}{2}E \cdot (E + K_S).
\end{equation}
Here $\CM_{E, \chi}$ stands for the moduli stack of pure 1-dimensional sheaves supported on $E \subset S$.
\medskip

We prove (\ref{eqn35}) by induction on the number of the irreducible components $\{E_i\}$.

For the induction base, we consider $E = nE'$ with $E'$ irreducible. Then $E' \simeq \BP^1$ and the normal bundle $\CO_{E'}(d)$ of $E'$ in $S$ satisfies
\begin{equation}\label{eqn36}
d= {E'}^2 = -2+E'\cdot(-K_S) >-2.
\end{equation}
Since the formal neighborhood of an irreducible toric divisor only depends on the degree of the normal bundle, the thickened curve $E =nE'\subset S$ is isomorphic to the $n$-th thickening
\[
nE' \subset \mathrm{Tot}(\CO_{E'}(d))
\]
of the 0-section in the total space of $\CO_{E'}(d)$. Hence by Proposition \ref{prop2.6} (where the condition (\ref{d>2g-2}) is guaranteed by (\ref{eqn36})), we have
\[
\mathrm{dim}\CM_{nE',\chi} \leq -n+\frac{1}{2}n(n-1)d = nE'\cdot (nE' +K_S).
\]
Here we used $E'^2=d$ and $E'\cdot K_S = -d-2$ in the last identity. This proves the induction base.

To complete the induction, we assume that $E = E' +E''$. Here
\[
E' = \sum_i n_iE'_i, \quad E''=\sum_im_iE''_i
\]
with $E'_i, E''_j$ irreducible toric divosrs satisfying $E'_i \neq E''_j$ for any $i,j$.

\begin{lem}\label{lem2.7}
Let $\CF$ be a pure 1-dimensional sheaf supported on $E$. Then there exists a canonical short exact sequence
\[
0 \to \CF' \to \CF \to \CF'' \to 0
\]
where $\CF'$ and $\CF''$ are pure 1-dimensional sheaves supported on $E'$ and $E''$ respectively.
\end{lem}

\begin{proof}
We take
\[
\CF'' : = \left( \CF|_{E''} \right)\slash \textup{maximal zero-dim subsheaf of } \CF|_{E''} \in \mathrm{Coh}(E'')
\]
and 
\[
\CF' := \mathrm{Ker}\left( \CF \twoheadrightarrow \CF|_{E''} \twoheadrightarrow \CF''\right) \in \mathrm{Coh}(E'). 
\]
Since $\CF'$ is a subsheaf of $\CF$, it is pure, supported on$E'$.
\end{proof}

Two sheaves $\CF'$ and $\CF''$ (as in Lemma \ref{lem2.7}) with different supports satisfy that
\[
\mathrm{Hom}_S(\CF'', \CF') = \mathrm{Hom}_{E'}(i^*\CF'', \CF') =0
\]
where $i: E' \hookrightarrow S$ is the embedding and $i^*\CF''$ is 0-dimensional on $E'$. Serre duality further implies
\[
\mathrm{Ext}^2_S(\CF'', \CF') = \mathrm{Hom}_S(\CF', \CF''\otimes \omega_S) =0.
\]
Hence by Lemma \ref{lem2.7}, after decomposing the stack $\CM_{E,\chi}$ into strata, we obtain a morphism to
\[
\bigsqcup_{\chi'+\chi''=\chi} \CM_{E',\chi'} \times \CM_{E'',\chi''}
\]
whose closed fiber over $(\CF', \CF'')$ has dimension upper bound 
\[
\mathrm{dim}\mathrm{Ext}^1(\CF'', \CF') = \chi(\CF'', \CF') = E' \cdot E''.
\]
Combining with the induction assumption on the dimensions of $ \CM_{E',\chi'}$ and $\CM_{E'',\chi''}$, we conclude that
\begin{align*}
    \mathrm{dim}\CM_{E,\chi} & \leq \frac{1}{2}E'\cdot(E'+K_S) +\frac{1}{2}E''\cdot(E''+K_S) + E'\cdot E''\\
    & = \frac{1}{2}E\cdot(E+K_S).
\end{align*}
This completes the induction. \qed.

\subsection{Proof of Theorem \ref{thm2.3}}\label{Sec2.6}

We first prove the irreducibility of $M^L_{\beta,\chi}$. Equivalently, we prove the irreducibility of the stack $\CM^L_{\beta,\chi}$. 

Recall the open subset $U \subset B$ formed by nonsingular curves in the linear system $|\beta|$. The open substack $h_\CM^{-1}(U)$ parameterizes line bundles on these curves with Euler characteristic $\chi$. In particular, $h_{\CM}^{-1}(U)$ is Zariski open and dense in an irreducible component of the relative Picard stack associated with the universal curve $\pi: \CC \to U$. Assume $\CM^L_{\beta,\chi}$ has another irreducible component $\CM'$ which does not contain $h_\CM^{-1}(U)$. By Lemma \ref{lem2.5} it has dimension 
\[
\mathrm{dim}\CM' = \beta^2,
\]
and it maps to the complement $B\setminus U$ under the morphism (\ref{stack_HC}). This implies that a general fiber of 
\[
h_{\CM}\big{|}_{\CM'} : \CM' \to B\setminus U
\]
has dimension at least 
\[
 \mathrm{dim}\CM' -(\mathrm{dim}B-1) \geq \beta^2-\frac{1}{2}\beta\cdot(\beta-K_S)+1 >\frac{1}{2}\beta\cdot (\beta+K_S),
\]
which contradicts Proposition \ref{prop2.5}. This completes the proof of the irreducibility of $M^L_{\beta,\chi}$. \qed


\subsection{Higgs bundles}
Most of the statements for $M^{L}_{\beta,\chi}$ discussed above hold identically for the moduli spaces $\widetilde{M}_{n,\chi}$ of Higgs bundles in the case of (B). This is due to the fact that $\widetilde{M}_{n,\chi}$ can be viewed as the moduli space of 1-dimensional semistable sheaves $\CF$ on $\mathrm{Tot}(\CO_C(D))$ with
\[
[\mathrm{supp}(\CF)]= n[C], \quad \chi(\CF) = \chi 
\]
via the spectral correspondence. We summarize these results in the following for the reader's convenience.

Recall the universal spectral curve 
\[
\pi: \CC_B \to B
\]
with $\pi_P: P =\mathrm{Pic}^0({\CC_B/B}) \to B$ the relative degree 0 Picard variety. Similar to the case of $M^L_{\beta,n}$, the group scheme $P$ acts on $\widetilde{M}_{n,\chi}$ via tensor product
\[
\CL \cdot \CF = \CL \otimes \CF, \quad \quad \CL\in \pi_P^{-1}(b),\quad \CF\in h^{-1}(b), \quad \forall b\in B.
\]
Here we view a Higgs bundle as a pure 1-dimensional coherent sheaf supported on the spectral curve
\[
\pi^{-1}(b) \subset \mathrm{Tot}(\CO_C(D)).
\]
The moduli stack of semistable Higgs bundles admits a morphism
\[
q: \widetilde{\CM}_{n,\chi} \to \widetilde{M}_{n,\chi}
\]
which induces
\[
h_\CM = h \circ q: \widetilde{\CM}_{n,\chi} \to B.
\]

\begin{prop}\label{prop2.9} Assume $\mathrm{deg}(D)=d>2g-2$.The following statements hold.
\begin{enumerate}
    \item[(1)] The fiber of $h_\CM$ over a closed point $b\in B$ satisfies
    \[
    \mathrm{dim}h_\CM^{-1}(b) \leq  n(g-1)+\frac{1}{2}n(n-1)d.
    \]
    \item[(2)] The stack $\widetilde{\CM}_{n,\chi}$ is irreducible and nonsingular of dimension 
    \[
    \mathrm{dim}\widetilde{\CM}_{n,\chi} =  n^2d.
    \]
    \item[(3)] The moduli space $\widetilde{M}_{n,\chi}$ is irreducible of dimension
    \[
    \mathrm{dim}\widetilde{M}_{n,\chi} =  n^2d+1.
    \]
    \item[(4)] The triple $(\widetilde{M}_{n,\chi}, P, B)$ form a weak abelian fibration of relative dimension
    \[
    g_n=  n(g-1)+\frac{1}{2}n(n-1)d+1.
    \]
\end{enumerate}
\end{prop}

\begin{proof}
These statements are parallel to Proposition \ref{prop2.6}, Lemma \ref{lem2.5}, Theorem \ref{thm2.3}, and Propsotion \ref{prop2.4}. (1) follows from a semicontinuity argument and Proposition \ref{prop2.6}. (2) follows from Serre duality for semistable Higgs bundles \cite[Corollary 2.6]{MSch}. (3) follows from (1,2) as explained in Section \ref{Sec2.6}. (4) is deduced by an identical proof as for Proposition \ref{prop2.4}.
\end{proof}

\subsection{Assumptions on the curve class $\beta$}
In Section 2, the ampleness assumption of the curve class $\beta$ is used for the following properties:
\begin{enumerate}
    \item[(1)] The linear system $|\beta|$ is base point free, and
    \item[(2)] a general curve in $|\beta|$ is integral and nonsingular.
\end{enumerate}
We may replace the ampleness assumption for $\beta$ by the conditions (1) and (2) above. 

\begin{prop}\label{prop2.10}
Theorem \ref{thm2.3} and Proposition \ref{prop2.4} hold for any curve class $\beta$ which contains an integral curve in the linear system $|\beta|$.
\end{prop}

\begin{proof}
Assume $C_0 \in |\beta|$ is integral. By the adjunction formula, either $C_0 \simeq \BP^1$ is an exceptional divisor or $C_0^2 >0$. In the first case, the moduli space is a reduced point. In the second case, we obtain that the divisor $C$ is integral and nef. Therefore (1,2) follow from \cite[Corollary 4.7]{DR} and the Bertini theorem.
\end{proof}

\section{Intersection cohomology complexes}\label{Sec3}

We prove in this section a support inequality for the moduli spaces $M^L_{\beta,\chi}$ and $\widetilde{M}_{n,\chi}$.

\begin{thm}\label{thm3.1}
Let $h: M \to B$ be the morphism (\ref{0.2_1}) or (\ref{0.2_2}). We define the $\delta$-function on $B$ from the associated group scheme $P$ as in Section \ref{Section_Supp}. Then any support $Z$ of $Rh_*\mathrm{IC}_M$ satisfies
\[
\mathrm{codim}Z \leq \delta_Z.
\]
\end{thm}

By Theorem \ref{thm1.8}, Proposition \ref{prop2.4}, and Proposition \ref{prop2.9} (4), it suffices to prove the following propositon concerning the intersection cohomology complex.

\begin{prop}\label{prop3.2}
We have
\begin{equation}\label{eqn3.2}
\tau_{>2R}\left( Rh_* \mathrm{IC}_M[-\mathrm{dim}M]\right) =0,\quad R:= \mathrm{dim}M-\mathrm{dim}B.
\end{equation}
\end{prop}

\subsection{Sketch of the proof of Proposition \ref{prop3.2}}

Altough Proposition \ref{prop3.2} only concerns bounded complexes on schemes, our proof relies on \emph{unbounded complexes} on \emph{Artin stacks}. From now on, we work with the derived category $D_c(-, \BQ_l)$ of constructible sheaves with $\BQ_l$-coefficients for Artin stacks as in \cite{LO1, LO2}. We denote by $D_c^b(-, \BQ_l)$,  $D^-(-, \BQ_l)$, and $D^+(-, \BQ_l)$ the subcategories of complexes which are bounded, bounded from above, and bounded from below respectively. In this section, we assume that all Artin stacks are of finite type. We use the six operations for Artin stacks following \cite{LO1, LO2, LO3}.  Furthermore, by \cite{LO3}, we also have the perverse $t$-structure in the unbounded setting.

Recall the morphsim from the moduli stack to the moduli space of 1-dimensional semistable sheaves/Higgs bundles 
\[
q: \CM \to M.
\]
The composition of $q$ and $h: M \to B$ induces a morphism
\[
h_\CM = h \circ q: \CM \to B.
\]
We consider the (unbounded) complexes 
\[
Rh_{\CM!} \BQ_{l} \in D^-(B, \BQ_l), \quad Rq_* \BQ_{l} \in D^+(M, \BQ_l).
\]

We first prove Proposition \ref{prop3.2} assuming the following two propositions which concern the stack $\CM$.

\begin{prop}\label{prop3.3}
We have
\begin{equation}\label{eqn3.3}
\tau_{>2R-2}\left(Rh_{\CM!} \BQ_{l}\right) =0.
\end{equation}
\end{prop}

\begin{prop}\label{prop3.4}
There exists a splitting
\begin{equation}\label{eqn38}
Rq_* \BQ_{l} \simeq \mathrm{IC}_{M}[-\mathrm{dim}M] \oplus 
\mathcal{E} \in D^+(M, \BQ_l).
\end{equation}
\end{prop}
\begin{proof}[Proof of Proposition \ref{prop3.2}]
Applying the dualizing functor to the isomorphism (\ref{eqn38}), we obtain
\[
\BD \left(Rq_* \BQ_{l}\right) \simeq \mathrm{IC}_M[\mathrm{dim}M] \oplus \mathcal{E}', \quad \mathcal{E}' \in D^-(M, \BQ_l).
\]
Since $\CM$ is nonsingular, the lefthand side is isomorphic to
\[
Rq_! \BD(\BQ_{l})=Rq_!\BQ_{l}[2\mathrm{dim} \CM]= Rq_!\BQ_{l}[2\mathrm{dim}M-2].
\]
Combining the two equations above, we conclude that
\[
Rq_! \BQ_{l} \simeq \mathrm{IC}_M[-\mathrm{dim}M+2] \oplus \cdots \in D^-(M, \BQ_l).
\]
Hence, thanks to properness of $h: M \to B$, we have $Rh_!= Rh_*$, and the lefthand side of (\ref{eqn3.2}) (shifted by degree 2) is a direct sum component of the lefthand side of (\ref{eqn3.3}). In particular, Proposition \ref{prop3.2} follows from Proposition \ref{prop3.3}.
\end{proof}

In the rest of Section 3, we prove Propositions \ref{prop3.3} and \ref{prop3.4}.

\subsection{Proof of Proposition \ref{prop3.3}.}

Proposition \ref{prop3.3} is a consequence of the following well-known vanishing and the dimension bounds (Proposition \ref{prop2.5} for (A) and Proposition \ref{prop2.9} (i) for (B)) obtained in Section 2.

\begin{lem}\label{Lem3.5}
Let $\CY$ be an irreducible Artin stack of dimension $r$. Then for $n > 2r$ we have the following vanishing for compactly supported cohomology:
\begin{equation}\label{eqnnn44}
H_c^n(\CY, \BQ_l) = 0.
\end{equation}
\end{lem}

\begin{proof}
In the special case when $\CY$ is nonsingular, (\ref{eqnnn44}) follows from the Verdier duality
\[
H_c^n(\CY, \BQ_l)^\vee = H^{2r-n}(\CY, \BQ_l) = 0, \quad 2r-n <0. 
\]
In general, since we only concern the constructible sheaf $\BQ_l$, we may assume that $\CY$ is reduced. Then by stratifying $\CY$ into locally closed nonsingular substacks and the excision sequences (\cite[Example 2.1 (iv)]{LO3}), we reduce (\ref{eqnnn44}) for general $\CY$ to the nonsingular ones.
\end{proof}

Let $b \in B$ be a closed point. We denote by $\CM_b$ the substack
\[
\CM_b : = h_\CM^{-1}(b) \subset \CM.
\]
Proposition \ref{prop2.5} and Proposition \ref{prop2.9} (i) yield
\[
\mathrm{dim}\CM_b \leq R-1, \quad \quad R= \mathrm{dim}M -\mathrm{dim}B.
\]
Combining with Lemma \ref{Lem3.5}, we conclude that the complex
\[
\left( Rh_{\CM!} \BQ_l \right)_b = H_c^*(\CM_b, \BQ_l)
\]
is concentrated in degrees $\leq 2(R-1)$ for any closed point $b \in B$. In particular, we have
\[
\Big{(}\tau_{>2R-2}(Rh_{\CM!} \BQ_l)\Big{)}_b =  \tau_{>2R-2} \Big{(} ( Rh_{\CM!} \BQ_l )_b \Big{)} =0, \quad \forall b\in B.
\]
This completes the proof of Proposition \ref{prop3.3}. \qed

\subsection{Moduli of framed objects}
The main difficulty for proving Proposition \ref{prop3.4} is the non-properness of the morphism $q: \CM \to M$. In order to apply the decomposition theorem \cite{BBD} to $q$, we use the moduli of framed objects \cite{Mein} to ``approximate" the stack $\CM$. We refer to \cite{D1, D3, D2, DM1, MR, MR1, MM} for applications of such techniques in the study of quivers representations and Donaldson--Thomas theory.

Let $M$ and $\CM$ be the moduli space and the moduli stack of (A) or (B) in Section \ref{Sec0.1}. Since in either case $M$ can be realized as a moduli space of semistable sheaves on an algebraic surface, we obtain (see Section \ref{Sec2.4}) that $M$ can be realized as a GIT-quotient of a Quot-scheme
\[
M = \mathrm{Quot}^{ss}\sslash \mathrm{GL}_m
\]
where the semistable locus $\mathrm{Quot}^{ss} \subset \mathrm{Quot}$ is with respect to a $\mathrm{GL}_m$-linearized polarization $\CL_m$ on $\mathrm{Quot}$. The morphism $q$ is induced by the morpshim from the stack to the corresponding good GIT-quotient: 
\begin{equation}\label{3.3_1}
 \CM = [\mathrm{Quot}^{ss}/ \mathrm{GL}_m] \xrightarrow{q}   \mathrm{Quot}^{ss}\sslash \mathrm{GL}_m = M.
\end{equation}

\begin{prop}\label{prop3.6}
For any $N >0$, there exist a nonsingular scheme $M_f$ and a nonsingular Artin stack $\CX_f$ with a commutative diagram
   \begin{equation}\label{eqn46}
    \begin{tikzcd}[column sep=small]
    M_f \arrow[dr, "p_M"] \arrow[rr, hook, "j"] & & \CX_f \arrow[dl, "p_\CX"] \\
       & \CM  & 
\end{tikzcd}
\end{equation}
satisfying the following properties:
\begin{enumerate}
     \item[(a)] $p_\CX$ is an affine space bundle,
     \item[(b)] $j: M_f \hookrightarrow \CX_f$ is an open immersion,
     \item[(c)] the composition $M_f \xrightarrow{p_M}\CM \xrightarrow{q} M$ is projective, and
    \item[(d)] for the complement $\CZ_f: = \CX_f \setminus M_f$, we have
    \[
    \mathrm{codim}_{\CX_f}\left( \CZ_f\right)>N.
    \]
\end{enumerate}
\end{prop}

\begin{proof}
We complete the proof of Proposition \ref{prop3.2} in the following 3 steps.
\medskip

\noindent {\bf Step 1: Quiver moduli.} For a fixed integer $f>m$, let $Q_f$ be a quiver of 2 vertices $P_1$ and $P_2$ such that the dimension vector is $(1,m)$ and there are $f$ arrows from $P_1$ to $P_2$. Following King \cite{King}, the representation space \[
\BA := \mathrm{Hom}(\BC, \BC^m)^f \simeq \BC^{mf}
\]
of the quiver $Q_f$ admits a natural action of the group 
\[
G_m := \mathrm{GL}_1 \times \mathrm{GL}_m.
\]
Moreover, for any $\theta >0$, the character 
\[
\chi_\theta: G_m \to \BC^*, \quad \quad (g_1,g_m) \mapsto \mathrm{det}(g_1)^{-m\theta}\cdot \mathrm{det}(g_m)^\theta, \quad g_i \in \mathrm{GL}_i
\]
yields a stability condition on $\BA$. Here the stability is given by GIT associated with the trivial line bundle $\CO^\theta_\BA$, equipped with the $G_m$-linearization induced by $\chi_\theta$. We denote by $\BA_\theta^{ss} \subset \BA$ the semistable locus with respect to $\theta$. 
\medskip

\noindent {\bf Claim.} We have
\[
\mathrm{codim}_{\BA}\left( \BA \setminus \BA_\theta^{ss} \right) \to \infty, \quad \quad \textup{when}~~f\to \infty.
\]
\begin{proof}[Proof of the claim]
If we view $\BA$ as the parameter space of $m\times f$ matrices, the GIT-unstable loci are contained in the determinantal variety $D_{m-1}\subset \BA$ formed by matrices of rank $<m$. Hence we have
\[
\mathrm{dim}(\BA \setminus \BA_\theta^{ss})\leq \mathrm{dim} D_{m-1} = (m-1)(f+1),
\]
which implies that
\[
\mathrm{codim}_{\BA}\left( \BA \setminus \BA_\theta^{ss} \right) \geq mf - (m-1)(f+1) = f+1-m \to \infty 
\]
when $f \to \infty$.
\end{proof}
\medskip

\noindent{\bf Step 2: Moduli of framed objects.} The moduli space of framed objects \cite{Mein} combines the quotients (\ref{3.3_1}) and the quiver $Q_f$, which provides the scheme $M_f$ and the stack $\CX_f$ for Proposition \ref{prop3.6}. We recall the construction as follows.

Consider the natural $G_m$-action on the product 
\[
\mathrm{Quot} \times \BA
\]
where the action on the first factor passes through the obvious $\mathrm{GL}_m$-action and the action on the second factor is given in Step 1. Since the diagonal torus
\[
\mathrm{GL}_1 \hookrightarrow \mathrm{GL}_1\times \mathrm{GL}_m = G_m
\]
acts trivially on both $\mathrm{Quot}$ and $\BA$, the $G_m$-action on $\mathrm{Quot} \times \BA$ further passes through a $\mathrm{PG}_m(:= G_m/\mathrm{GL}_1)$-action.

We choose $a>0$ such that $a > m\theta$, and we consider the $G_d$-linearization
\begin{equation}\label{3.3_2}
\CL^{a,\theta} : = \CL_m^{\otimes a} \boxtimes \CO^\theta_{\BA}
\end{equation}
on $\mathrm{Quot}\times \BA$. Let
\[
(\mathrm{Quot}\times \BA)^{ss} \subset   \mathrm{Quot}\times \BA
\]
be the GIT-semistable locus associated with (\ref{3.3_2}). By a calculation using the Hilbert--Mumford criterion, it was proven in \cite{Mein} (under a more general setup) that we have the open immersions
\begin{equation}\label{eqnnn48}
   \mathrm{Quot}^{ss} \times \BA^{ss}_\theta \subset (\mathrm{Quot}\times \BA)^{ss}  \subset \mathrm{Quot}^{ss} \times \BA.
\end{equation}
Here the first inclusion is given in \cite[Remark 3.40]{Mein} and the second inclusion is given in \cite[Proposition 3.39]{Mein}.

We define
\begin{equation}\label{3.3_3}
M_f :=  (\mathrm{Quot}\times \BA)^{ss}\sslash G_m = (\mathrm{Quot}\times \BA)^{ss} /\mathrm{PG}_m, \quad \CX_f := (\mathrm{Quot}^{ss}\times \BA)/ \mathrm{PG}_m.
\end{equation}
Note that by \cite[Proposition 3.39]{Mein}, the scheme $M_f$ can be interpreted as the moduli of framed objects which is nonsingular by \cite[Corollary 3.41]{Mein}.

\medskip
\noindent{\bf Step 3: Properties.} We show that $M_f$ and $\CX_f$ defined in (\ref{3.3_3}) fit into the commutative diagram (\ref{eqn46}) and satisfy the properties (a,b,c). Moreover, they also satisfy (d) when $f \to \infty$.

The open immersion
\[
j: M_f \hookrightarrow \CX_f
\]
is induced by the second inclusion of (\ref{eqnnn48}). The quotient stack $\CX_f$ admits a natural map to $\CM$ via the natural projection
\[
p_\CX: \CX_f =(\mathrm{Quot}^{ss}\times \BA)/ \mathrm{PG}_m \to \mathrm{Quot}^{ss}/ \mathrm{PG}_m  = \CM
\]
which is an $\BA$-bundle. By setting $p_M = p_\CX\circ j$, we obtain the commutative diagram (\ref{eqn46}) and (a,b) immediately. The property (c) follows from \cite[Theorem 3.42]{Mein}. 

We note that the morphism from $M_f$ to $M$ has a natural geometric interpretation. In fact, we have
\[
M = (\mathrm{Quot}^{ss} \times \BA)\sslash G_m
\]
by \cite{Mein} the last paragraph before Section 3.7. Therefore the contraction
\[
q\circ p_M: M_f = (\mathrm{Quot}\times \BA)^{ss}\sslash G_m \to (\mathrm{Quot}^{ss} \times \BA)\sslash G_m=  M
\]
can be viewed as a variation of GIT from the $G_m$-linearization (\ref{3.3_1}) to the $G_m$-linearization
\[
\CL^{a,0} : = \CL_m^{\otimes a} \boxtimes \CO_\BA
\]
with the trivial $G_m$-action on the second factor.

It remains to prove (d). By (\ref{eqnnn48}) and the claim in Step 1, we have
\[
\mathrm{codim}_{\CX_f}(\CZ_f) \geq \mathrm{codim}_{\BA}(\BA \setminus \BA_{\theta}^{ss}) \to \infty, \quad \quad \textup{when}~~f\to \infty.
\]
This completes the proof.
\end{proof}

\subsection{Proof of Proposition \ref{prop3.4}}

To construct the desired splitting, we follow the approach of \cite{Mein}. Namely we use the construction in the previous section to approximate $\CM$ with $M_f$ and apply the decomposition theorem to the proper morphism $q\circ p_M: M_f \rightarrow M$.

Fix $N > 0$ and choose $f$ as in Proposition \ref{prop3.6}.
Let $i: \CZ_f \hookrightarrow M_f$ denote the closed immersion which has codimension larger than $N$.
Consider the excision triangle on $\CX_f$ 
\begin{equation}\label{triangle}
i_{!}i^{!}\BQ_l \rightarrow \BQ_l \rightarrow Rj_*j^*\BQ_l\rightarrow
i_{!}i^{!}\BQ_l [1].
\end{equation}
Since $\CX_f$ is nonsingular, we have 
\[
i^!\BQ_l  = \omega_{\CZ_f}[-2\mathrm{dim}\CX_f];
\]
also, from Section V.2 of \cite{Iversen}, we have that the complex $\omega_{\CZ_f}$ is concentrated in degrees $[-2\mathrm{dim}\CZ_f, \infty]$.   By combining these with the codimension bound for $\CZ_f$,  we have that the complex $i^!\BQ_l$ is supported in degrees $[2N, \infty]$.

If we pushforward the excision triangle (\ref{triangle}) to $M$ along 
\[
q\circ p_{\CX}: \CX_f \rightarrow M,
\]
we obtain the triangle
\begin{equation}\label{eqn51}
R(q\circ p_{\CX})_*i_{!}i^{!}\BQ_l \rightarrow R(q\circ p_{\CX})_*\BQ_l \rightarrow R(q\circ p_{\CX})_*Rj_*j^*\BQ_l\rightarrow R(q\circ p_{\CX})_*i_{!}i^{!}\BQ_l[1].
\end{equation}
Since the derived pushforward functor preserves the subcategory $D^{\geq 0}$ and its shifts, the leftmost term of (\ref{eqn51}) is supported in degrees $[2N, \infty]$ as well.   Furthermore, by \cite[Lemma 3.3]{LO3}, after changing $N$ by a bounded amount, we have the same support result for perverse cohomology sheaves.  In other words, we have the vanishing
\[
^{\mathfrak{p}}\tau_{\leq 2N} R(q\circ p_{\CX})_*i_{!}i^{!}\BQ_l  = 0
\]
and, by (\ref{eqn51}), the quasi-isomorphism
\begin{equation}\label{eqn52}
^{\mathfrak{p}}\tau_{< 2N} R(q\circ p_{\CX})_*\BQ_l \xrightarrow{\sim } {^\mathfrak{p}}{\tau_{< 2N}} R(q\circ p_{\CX})_*Rj_*j^*\BQ_l.
\end{equation}

Finally, since $p_\CM: \CX_f \rightarrow \CM$ is an affine space bundle, so that ${Rp_{\CX}}_*\BQ_l = \BQ_l$, we can rewrite (\ref{eqn52}) as
\begin{equation}\label{truncation}
^{\mathfrak{p}}\tau_{< 2N} Rq_*\BQ_l  \xrightarrow{\sim} {^\mathfrak{p}}{\tau_{< 2N}} R(q\circ p_{\CM})_*\BQ_l.
\end{equation}

By the decomposition theorem for the proper, surjective morphism $q\circ p_{\CM}: M_f \rightarrow M$, the right-hand side of (\ref{truncation}) can be non-canonically written as a direct sum of its (shifted) perverse cohomology sheaves:
\begin{equation}\label{splitting}
^{\mathfrak{p}}{\tau_{< 2N}} R(q\circ p_{\CM})_*\BQ_l \simeq \bigoplus_{k = \mathrm{dim}M}^{2N-1} \CP_{k}[-k].
\end{equation}
The lowest perverse cohomology sheaf $\CP_{\mathrm{dim}M}$ occurs in degree $\mathrm{dim}M$, because of surjectivity of the morphism. After restricting to an open subset $V \subset M$ over which $q\circ p_\CM$ is smooth, it is given by the shifted local system whose fiber over $x \in V$ is $H^0(M_{f,x}, \BQ_l)$ with $M_{f,x}$ the closed fiber of $M_f$ over $x$.  In particular, $\CP_{\mathrm{dim}M}$ contains $\mathrm{IC}_M$ as a direct summand.

If we combine the splitting (\ref{splitting}) with \eqref{truncation}, we see that the composition
\[
\mathrm{IC}_M[-\mathrm{dim}M] \rightarrow \CP_{\mathrm{dim}M}[-\mathrm{dim}M] \rightarrow {^\mathfrak{p}}{\tau_{< 2N}} Rq_*\BQ_l
\]
admits a splitting
\[
u_N: {^\mathfrak{p}}{\tau_{< 2N}} Rq_*\BQ_l \rightarrow \mathrm{IC}_M[-\mathrm{dim}M].
\]

As Ext-groups between perverse sheaves vanish outside of bounded degree, for $N$ sufficiently large, we have the stabilization
\[
\operatorname{Hom}\left(^{\mathfrak{p}}\tau_{< 2N} Rq_*\BQ_l, \mathrm{IC}_M[-\mathrm{dim}M]\right)
 = \operatorname{Hom}\left(^{\mathfrak{p}}\tau_{< 2N+1} Rq_*\BQ_l, \mathrm{IC}_M[-\mathrm{dim}M]\right).
 \]
In other words, for the canonical morphism 
\[
{v_N:} {^{\mathfrak{p}} \tau_{< 2N}} Rq_*\BQ_l \rightarrow {^\mathfrak{p}}{\tau_{< 2(N+1)}} Rq_*\BQ_l,
\]
the splittings $u_N$ are compatible in the sense that 
$u_{N} = u_{N+1}\circ v_{N}$.

By \cite[Lemma 4.3.2]{LO1}, we have that the unbounded complex $Rq_*\BQ_l$ is the homotopy colimit of its truncations, \emph{i.e.},
\[
Rq_*\BQ_l = \operatorname{hocolim}_{N \rightarrow \infty} {^\mathfrak{p}}{\tau_{< 2N}} Rq_*\BQ_l.
\]
So as a result, the splittings $u_N$ yield a splitting
$$u: Rq_*\BQ_l \rightarrow \mathrm{IC}_M[-\mathrm{dim}M],$$
and consequently, a direct summand decomposition (\ref{eqn38}) as desired.
\qed

\section{Proof of the main theorem}
\subsection{Overview}
We complete the proof of Theorem \ref{thm0.2}. For the approach, we combine the support inequality of Theorem \ref{thm3.1} and techniques of \cite{CL, dC_SL}.

\subsection{$\delta$-inequalities for integral curves}\label{Sec4.2}
As in Section \ref{Sec2.2}, we consider the relative degree 0 Picard variety
\[
\mathrm{Pic}^0({\CC_B/B}) \to B
\]
associated with a family of curves $\pi_B: \CC_B \to B$. We obtain the $\delta$-invariants computing the dimensions of the affine parts of the group schemes 
\[\delta(b): = \mathrm{dim}\left(\mathrm{Pic}^0(\CC_b)^{\mathrm{aff}}\right), \quad b\in B,
\]
where $\CC_b$ is the fiber of $\pi_B$ over $b$. For a closed subvariety $Z\subset B$, we define $\delta_Z$ to be $\delta(b)$ for a general point in $Z$.

In Section \ref{Sec4.2}, we first focus on the case of a flat family of \emph{integral} curves
\[
\pi_B: \CC_B \to B
\]
satifying that the compatified Jacobian 
\[
\mathrm{Pic}^0({\CC_B/B}) \subset \overline{\CJ\CC}_B
\]
(parameterizing degree 0 torsion free sheave on the curves $\CC_b$) is nonsingular.

The following lemma is the ``Severi inequality" which is parallel to \cite[(41)]{dC_SL} for Higgs bundles. The proof of \cite[(41)]{dC_SL} works identically here since the only structure used for the spectral curves $\CC_B \to B$ in \cite{dC_SL} is the smoothness of $\overline{\CJ\CC}_B$. We give a proof for the reader's convenience.

\begin{lem}\label{lem4.1}
For any subvariety $Z\subset B$, we have
\[
\mathrm{codim}Z \geq \delta_Z.
\]
\end{lem}

\begin{proof}
Let $\CC_{\mathrm{univ}} \to B_{\mathrm{univ}}$ be a semi-universal family of curves such that the family $\CC_B \to B$ is induced by a map 
\[
\phi:B \to B_{\mathrm{univ}}.
\]
Let $B_{\mathrm{univ}}^\delta \subset B_{\mathrm{univ}}$ be the locus given by $\{b\in B: \delta(b) = \delta\}$. Since $\overline{\CJ\CC}_B$ is nonsingular, by the paragraph following \cite[Theorem 2]{FGvS}, the image $\phi(B) \subset B_{\mathrm{univ}}$ meets $B_{\mathrm{univ}}^\delta$ transversally. Hence for any irreducible subvariety $Z \subset B$ whose general points lie in $B_{\mathrm{univ}}^\delta$, we have
\[
 \mathrm{dim}Z \leq \mathrm{dim}(\phi(B) \cap B_{\mathrm{univ}}^\delta) = \mathrm{dim}\phi(B) + \mathrm{dim}B_{\mathrm{univ}}^\delta - \mathrm{dim}B_{\mathrm{univ}} \leq \mathrm{dim}B-\delta
\]
where the equality follows from the transversality. We conclude that 
\[
\mathrm{codim}Z = \mathrm{dim}B-\mathrm{dim}Z \geq \delta = \delta_Z. \qedhere
\]
\end{proof}

Our major application of Lemma \ref{lem4.1} is for curves in a linear system on a del Pezzo surface.

\begin{cor}\label{cor4.2}
Let $\beta$ be a curve class on a del Pezzo surface, let 
\begin{equation}\label{eqn_cor4.2}
\pi_B: \CC_B \to B=\BP H^0(S, \CO_S(\beta))
\end{equation}
be the universal curve in the linear system, and let $\pi^\circ_B: \CC^\circ \to B^\circ$ be the restriciton of (\ref{eqn_cor4.2}) to the subset $B^\circ \subset B$ of integral curves. Then for any irreducible subvariety $Z\subset B$ whose generic point lies in $B^\circ$, we have
\[
\mathrm{codim}Z \geq \delta_Z.
\]
\end{cor}

\begin{proof}
Since the compactified Jacobian $\overline{\CJ\CC}_{\CC^\circ}$ associated with $\pi^\circ_B: \CC^\circ \to B^\circ$ is an open subvariety of the moduli of stable pure 1-dimensional sheaves supported in the class $\beta$, we deduce the smoothness of $\overline{\CJ\CC}_{\CC^\circ}$ from the smoothness of the moduli stack (\emph{c.f.} Lemma \ref{lem2.5}). Hence Corollary \ref{cor4.2} follows by applying Lemma \ref{lem4.1} to the family $\pi^\circ_B: \CC^\circ \to B^\circ$.
\end{proof}

When the integral locus $B^\circ \subset B$ is nonempty, by Proposition \ref{prop2.10} the moduli space $M^L_{\beta,\chi}$ is irreducible for any polarization $L$ and $\chi \in \BZ$ satisfying that 
\[
\mathrm{dim}M^L_{\beta,\chi} = \mathrm{dim} \overline{\CJ\CC}_{B^\circ} = \beta^2+1.
\]
We define the following invariant associated with the curve class $\beta$:
\begin{equation}\label{Phi}
\Phi_\beta : = \mathrm{dim}\mathrm{Pic}^0({\CC^\circ/B^\circ})-2\mathrm{dim}B = \mathrm{dim}M^L_{\beta,\chi} -2\mathrm{dim}B = 1+\beta\cdot K_S
\end{equation}
where the last eqaulity follows from Lemma \ref{lem2.1}.

\subsection{$\delta$-inequalities for linear systems}\label{sec4.3}
In Section \ref{sec4.3}, we assume that $\beta$ is an ample curve class on a del Pezzo surface $S$.  We introduce a stratification of 
\[
B = \BP H^0(S, \CO_S(\beta)) 
\]
analogous to the stratification introduced in \cite[Section 9]{CL} and \cite[Section 5.2]{dC_SL} for Higgs bundles.

We consider the following $s$-tuples
\begin{equation}\label{type}
\underline{\beta} =\big{(}(m_1, \beta_1), (m_2, \beta_2), \dots, (m_s, \beta_s) \big{)}
\end{equation}
where $s\geq 1$, $m_i\geq 1$, and $\beta_i$ are (not necessarily distinct) curve classes on $S$ satisfying
\begin{enumerate}
    \item[(1)] $\sum_{i=1}^s m_i\beta_i = \beta$, and 
    \item[(2)] there exists an integral curve in $|\beta_i|$ for each $1\leq i \leq s$.
\end{enumerate}
The objects (\ref{type}) are called \emph{types} of the curves in the linear system $B=|\beta|$. We say that 
\[
\underline{\beta} =\big{(}(m_1, \beta_1), (m_2, \beta_2), \dots, (m_s, \beta_s) \big{)} \quad \mathrm{and}\quad \underline{\beta}' =\big{(}(m'_1, \beta'_1), (m'_2, \beta'_2), \dots, (m'_s, \beta'_{s'}) \big{)}
\]
give the same type, if $s=s'$, and there exists a bijection
\[
\sigma: \{ 1,2,\dots, s \} \rightarrow \{1,2,\dots, s\}
\]
such that $\beta_i = \beta'_{\sigma(i)}$ and $m_i = m'_{\sigma(i)}$. We have a stratification according to the types of the curves in $|\beta|$:
\[
B = \bigsqcup_{\underline{\beta}} B_{\underline{\beta}}
\]
where each $B_{\underline{\beta}}$ is a locally closed subset of $B$ formed by curves in $|\beta|$ of type $\underline{\beta}$:
\[
B_{\underline{\beta}} = \left\{ E= \sum_im_iE_i \in |\beta|,~~ E_i \in |\beta_i|,~~~ E_i \textup{~~are~~distinct~~integral~~curves} \right\}.
\]

\begin{prop}\label{prop4.3}
Let $Z\subset B$ be an irreducible subvariety whose general points have type 
\[
\underline{\beta} = \big{(} (m_1, \beta_1), (m_2, \beta_2), \dots, (m_s, \beta_s) \big{)}. 
\]
Then we have
\[
\Phi_\beta +\mathrm{codim}Z \geq \sum_{i=1}^s \Phi_{\beta_i}+\delta_Z.
\]
\end{prop}

\begin{proof}
We apply a similar argument as in \cite[Corollary 5.4.4]{dC_SL} for Higgs bundles. 

For a curve class $\beta_i$, we denote by $|\beta_i|^\circ$ the open subvariety of $|\beta_i|=\BP H^0(S, \CO_S(\beta))$ consisting of integral curves. We define
\begin{equation}\label{eqnn44}
\CC^\circ_{\beta_i} \to |\beta_i|^\circ, \quad \mathrm{Pic}_{\beta_i} \to |\beta_i|^\circ
\end{equation}
to be the universal curve and the corresponding relative degree 0 Picard variety over $|\beta_i|^\circ$. For a type $\underline{\beta}$ as in (\ref{type}), we have a finite morphism
\[
\lambda_{\underline{\beta}}: B'_{\underline{\beta_i}}:= \prod_{i=1}^s |\beta_i| \rightarrow B, \quad \quad (E_i)_{i=1}^s \mapsto \sum_{i=1}^s m_iE_i
\]
whose image is
\[
\mathrm{Im}(\lambda_{\underline{\beta}}) = \overline{{B_{\underline{\beta_i}}}} \subset B.
\]
The morphism $\lambda_{\underline{\beta}}$ sends the open subvariety $\prod_{i=1}^s |\beta_i|^\circ \subset B'_{\underline{\beta_i}}$ to ${{B_{\underline{\beta_i}}}} \subset B$.

Now we assume that $\eta \in B$ is the generic point of $Z$. By the first two paragraphs of \cite[Proof of Corollary 5.4.4]{dC_SL}, there exists a point 
\[
\eta' = (\eta_1,\eta_2,\dots, \eta_s)\in \lambda_{\underline{\beta}}^{-1}(\eta) \subset B'_{\underline{\beta_i}}, \quad \eta_i \in |\beta_i|^\circ
\]
satisfying that 
\begin{enumerate}
    \item[(i)] $\mathrm{dim}\overline{\{\eta\}} = \mathrm{dim}\overline{\{\eta'\}} = \mathrm{dim}Z $, 
    \item[(ii)] $\mathrm{dim}(\mathrm{Pic}_{\beta,\eta})^\mathrm{ab} = \sum_{i=1}^s\mathrm{dim}(\mathrm{Pic}_{\beta_i,\eta_i})^\mathrm{ab}$.
\end{enumerate}
Here for any connected commutative group scheme $P$, we use the notation $P^\mathrm{ab}$ to denote its abelian variety part in the Chevalley decomposition (\ref{Chevalley}). Since by definition (see (\ref{Phi})) we have
\[
\mathrm{dim}(\mathrm{Pic}_{\beta,\eta})^\mathrm{ab} = \Phi_\beta+\mathrm{dim}B-\delta_Z, \quad  \mathrm{dim}(\mathrm{Pic}_{\beta_i,\eta_i})^\mathrm{ab} = \Phi_{\beta_i}+\mathrm{dim}|\beta_i|-\delta_{\overline{\{\eta_i\}}},
\]
we obtain from (ii) that
\begin{equation}\label{eqn44}
\delta_Z + \sum_{i=1}^s\Phi_{\beta_i} -\Phi_\beta = \mathrm{dim}B -\sum_{i=1}^s\mathrm{dim}|\beta_i|+\sum_{i=1}^s \delta_{\overline{\{\eta_i\}}}.
\end{equation}

Applying Corollary \ref{cor4.2} to (\ref{eqnn44}), we have
\[
\delta_{\overline{\{\eta_i\}}} \leq \mathrm{dim}|\beta_i| - \mathrm{dim}\overline{\{\eta_i\}},
\]
which implies that the righthand side of (\ref{eqn44}) satisfies
\[
\mathrm{RHS~~of~~(\ref{eqn44})} \leq \mathrm{dim}B - \sum_{i=1}^s\mathrm{dim}\overline{\{\eta_i\}}= \mathrm{dim}B-\mathrm{dim}Z =\mathrm{codim}Z.
\]
This completes the proof.
\end{proof}

\subsection{Higgs bundles} The analog of Proposition \ref{prop4.3} for the moduli of Higgs bundles is exactly the inequality (75) of \cite[Corollary 5.4.4]{dC_SL}. Altough the paper \cite{dC_SL} works with Higgs bundles with $\mathrm{gcd}(n, \chi)=1$, Corolloary 5.4.4 only concerns the group scheme  --- the degree 0 Picard variety associated with the universal spectral curve, which is not constrained by the coprime assumption.

We now rewrite the inequality (75) of \cite[Corollary 5.4.4]{dC_SL} in Proposition \ref{prop4.4} parallel to the form of Proposition \ref{prop4.3}.

Consider the Hitchin fibration 
\[
h: \widetilde{M}_{n,\chi} \rightarrow B
\]
associated with $C$, $n$, $\chi$, and an effective divisor $D$ with degree $\mathrm{deg}(D)>2g-2$. Recall from \cite[Section 5.2]{dC_SL} that the Hitchin base
\[
B = \bigoplus_{i=1}^nH^0(C, \CO(iD))
\]
admits a stratification
\begin{equation}\label{strat}
B = \bigsqcup_{(n_\bullet, m_\bullet)} B_{n_\bullet, m_\bullet}.
\end{equation}
Here a type of spectral curves is given by $(n_\bullet, m_\bullet)$ with 
\[
s\geq 1, \quad n_\bullet = (n_1, n_2, \dots ,n_s), \quad m_\bullet =(m_1, m_2\ \dots, m_s), \quad \sum_{i=1}^s m_in_i = n,
\]
and $B_{n_\bullet, m_\bullet}$ are formed by spectral curves $E \subset \mathrm{Tot}(\CO_C(D))$ of the form $E= \sum_im_iE_i$ with $E_i$ distinct integral spectral curves which are degree $n_i$ covers of the zero section $C$. This actually coincides with the notion of (\ref{type}) since the class of any spectral curve in the surface $\mathrm{Tot}(\CO_C(D))$ is of the form $\beta_i = n_i [C]$ with $[C]$ the curve class of the zero section $C \subset \mathrm{Tot}(\CO_C(D))$. We refer to \cite[Section 5]{dC_SL} for more detials about the stratification (\ref{strat}).

We define the following invariant similar to (\ref{Phi}):
\[
\Phi_n = \mathrm{dim}\widetilde{M}_{n,\chi} - 2\mathrm{dim}B = 1+(2g-2-\mathrm{deg}(D))n
\]
where we use the dimension formulas of \cite[(77)]{dC_SL} in the last identity.

\begin{prop}\label{prop4.4}
Let $Z\subset B$ be an irreducible subvariety whose general points have type $(n_\bullet, m_\bullet)$. Then we have
\[
\Phi_n +\mathrm{codim}Z \geq \sum_{i=1}^s \Phi_{n_i}+\delta_Z.
\]
Here $\delta_Z$ is defined via the relative degree 0 Picared variety associated with the spectral curves.
\end{prop}

\subsection{Proof of Theorem \ref{thm0.2}}
We complete the proof of Theorem \ref{thm0.2} in this section. Let 
\[
h: M^L_{\beta,\chi} \rightarrow B
\]
be the morphism (\ref{0.2_1}). By Lemma \ref{lem2.5}, the open subvariety of stable sheaves 
\[
M^{L,s}_{\beta,\chi} \subset M^L_{\beta,\chi}
\]
is nonsingular. So we have 
\[
\left(\mathrm{IC}_{M^{L}_{\beta,\chi}}\right)\big{|}_{M^{L,s}_{\beta,\chi}} = \BQ[\mathrm{dim}M^{L}_{\beta,\chi}].
\]
In particular, the restriction of the direct image complex $Rh_*\mathrm{IC}_{M^{L}_{\beta,\chi}}$ to the open subset $U\subset B$ of nonsingular curves in $|\beta|$ satisfies
\begin{equation}\label{eqn47}
Rh_*\mathrm{IC}_{M^{L}_{\beta,\chi}}\big{|}_U \simeq  \bigoplus_{i=0}^{2d}\wedge^i R^1\pi_* \BQ[\mathrm{dim}M^{L}_{\beta,\chi}-i].
\end{equation}
Here $\pi: \CC \to U \subset B$, and (\ref{eqn47}) is an isomorphism of variations of Hodge structures by Proposition \ref{prop2.2}. Hence, in order to prove Theorem \ref{thm0.2}, it suffices to show that the lefthand side of (\ref{0.2_3}), as a bounded complex of perverse sheaves, has full support $B$.

Assume that the irreducible subvariety $Z \subset B$ is a support whose general point has type $\underline{\beta}$. We combine the inequalities of Proposition \ref{prop4.3} and Theorem \ref{thm3.1}:
\[
\Phi_\beta +\mathrm{codim}Z \geq \sum_{i=1}^s \Phi_{\beta_i}+\delta_Z \geq \sum_{i=1}^s \Phi_{\beta_i} + \mathrm{codim}Z,
\]
which implies $\Phi_\beta \geq \sum_{i=1}^s \Phi_{\beta_i}$. Therefore we obtain that 
\begin{equation}\label{eqn233}
1-s \geq (\beta -\sum_i \beta_i)\cdot(-K_S). 
\end{equation}
Since $-K_S$ is ample and $\beta -\sum_i \beta_i \geq 0$, the only possibility for (\ref{eqn233}) to hold is $s=1$ and $m_i=1$. Equivalently, we have $Z= B$. This complete the proof of Theorem \ref{thm0.2} for $M^L_{\beta,\chi}$.

The proof for $\widetilde{M}_{n,\chi}$ is completely identical where we apply Theorem \ref{thm3.1} and Proposition \ref{prop4.4}. \qed


\end{document}